%% file: main.tex
\documentclass[11pt]{article}
\usepackage{amssymb}
\usepackage{fancybox}
\usepackage{mathrsfs}
\usepackage{amsrefs}
\usepackage{mathtools}
\usepackage{yhmath}
\usepackage{subcaption}
\usepackage{amscd}
\usepackage{amsmath}
\usepackage{amsthm}
\usepackage{mymacros}
\usepackage{enumitem}

\usepackage{comment}
\usepackage[all]{xy}
\usepackage{xypic}
\xyoption{line}   
\usepackage[margin=3cm]{geometry}
\usepackage{graphicx}
\usepackage{eucal}
\usepackage{bm}
\usepackage{cancel}
\usepackage{color}
\usepackage{accents}

\usepackage{hyperref}
 \usepackage{scalerel,stackengine}

\theoremstyle{plain}
\newtheorem{thm}{Theorem}[section]
\newtheorem{prop}[thm]{Proposition}
\newtheorem{lem}[thm]{Lemma}
\newtheorem{defilem}[thm]{Definition-Lemma}
\newtheorem{cor}[thm]{Corollary}

\theoremstyle{definition}
\newtheorem{defi}[thm]{Definition}

\newtheorem{rem}[thm]{Remark}

\newcommand{\ZZ}{\mathbb{Z}}
\newcommand{\calA}{\mathcal{A}}

\newcommand{\QQ}{\mathbb{Q}}

\newcommand{\CC}{\mathbb{C}}

\newcommand{\CS}{{\mathbb{C}^{*}}}

\newcommand{\GR}{\mathrm{Gr}}

\DeclareMathOperator{\Hom}{Hom}

 \newcommand{\wt}[1]{\widetilde{#1}}

\newcommand{\ov}[1]{\overline{#1}}

\newcommand{\DR}{\mathrm{R}}

\newcommand{\simto}{\overset{\sim}{\to}}

\newcommand{\lra}{\longrightarrow}

\newcommand{\calS}{\mathcal{S}}

\newcommand{\calM}{\mathcal{M}}

\newcommand{\calC}{\mathscr{C}}

\newcommand{\calL}{\mathcal{L}}

\newcommand{\calG}{\mathcal{G}}

\newcommand{\Ker}{\mathrm{Ker}}
\newcommand{\Ima}{\mathrm{Im}}

\newcommand{\Var}{\mathrm{Var}}
\newcommand{\Can}{\mathrm{can}}

\newcommand{\MHM}{\mathrm{MHM}}

\newcommand{\iden}{\mathrm{id}}

\newcommand{\kaa}{\mathrm{Ker}}
\newcommand{\cokaa}{\mathrm{Coker}}

\newcommand{\bld}[1]{\boldsymbol{#1}}

\newcommand{\lag}{\langle}
\newcommand{\rag}{\rangle}
\stackMath
\newcommand\reallywidehat[1]{%
\savestack{\tmpbox}{\stretchto{%
  \scaleto{%
    \scalerel*[\widthof{\ensuremath{#1}}]{\kern-.6pt\bigwedge\kern-.6pt}%
    {\rule[-\textheight/2]{1ex}{\textheight}}%WIDTH-LIMITED BIG WEDGE
  }{\textheight}% 
}{0.5ex}}%
\stackon[1pt]{#1}{\tmpbox}%
}
\newcommand{\Id}{\mathrm{id}}

\newcommand{\diag}{\mathrm{Diag}}

\newcommand{\Derived}{\mathrm{R}}

\makeatletter
\newcommand{\subjclass}[2][2020]{%
  \let\@oldtitle\@title%
  \gdef\@title{\@oldtitle\footnotetext{#1 \emph{Mathematics subject classification.} #2}}%
}
\newcommand{\keywords}[1]{%
  \let\@@oldtitle\@title%
  \gdef\@title{\@@oldtitle\footnotetext{\emph{Key words and phrases.} #1.}}%
}
\makeatother

\title{Classification of can-Var diagrams of mixed Hodge structures and monodromic mixed Hodge modules on $\CC$}
 \date{}
\author{Takahiro Saito}

\subjclass{14F10, 32S35, 32S40}
\keywords{D-module, Perverse sheaf, Mixed Hodge module}

\begin{document}
\maketitle

\begin{abstract}
By our previous work, 
the category of unipotent monodromic mixed Hodge modules on $\CC$ is equivalent to the category of can-Var diagrams of mixed Hodge structures.
In this paper, we show that the associated graded diagram with respect to the weight filtrations decomposes as a finite direct sum of four basic types tensored with pure Hodge structures,
which we call a split graded model.
We then show that every unipotent monodromic mixed Hodge module is represented by a split graded model together with a finite family of $\CC$-linear maps.
We also obtain an analogous description in the case where the monodromy has no eigenvalue $1$.
\end{abstract}
\tableofcontents

\section{Introduction}

The category $\MHM(X)$ of graded-polarizable mixed Hodge modules on a complex manifold $X$ was introduced by M. Saito~\cite{HM88}, \cite{MHM} in order
to provide a functorial framework for Hodge theory.
Roughly speaking, its objects are regarded as extensions of variations of mixed Hodge structure defined on open dense subsets of their supports.
While the general theory of mixed Hodge modules is well established,
mixed Hodge modules themselves are still complicated in general
and difficult to describe explicitly.
Throughout this paper, 
mixed Hodge structures are always assumed to be graded-polarizable,
and mixed Hodge modules are always assumed to be algebraic and graded-polarizable.
For simplicity, we will often omit these adjectives.

In \cite{SaiMon} and \cite{SaiHodgefilt},
we studied how a class of mixed Hodge modules called monodromic mixed Hodge modules can be described.
For an affine line $\CC_t$ with coordinate $t$,
let $\calM$ be a mixed Hodge module $(M,F^\bullet M,K,W_\bullet K)$,
where $M$ is the underlying $D$-module, $F^\bullet M$ is the Hodge filtration,
$K$ is the underlying $\QQ$-perverse sheaf and
$W_\bullet K$ is the weight filtration.
If the restriction $K|_{\CS}$ of $K$ to $\CS:=\CC_t\setminus \{0\}$ is a shifted local system,
then we say that $\calM$ is monodromic.
We denote by $\MHM_{\mathrm{mon}}(\CC_t)$ the category of 
monodromic mixed Hodge modules on $\CC_t$.
We consider the tuple 
\[(\psi_{t}\calM,\phi_{t,1}\calM,\Can,\Var).\]
Here
$\psi_{t}\calM$ is the nearby cycle of $\calM$ along $t=0$,
regarded as a mixed Hodge structure,
equipped with the semisimple part $T_s$ of its monodromy automorphism and the nilpotent morphism 
$N=\frac{-1}{2\pi\sqrt{-1}}\log T_u\colon \psi_{t}\calM\to \psi_{t}\calM(-1)$,
where $T_u$ is the unipotent part of this monodromy automorphism.
Moreover, $\phi_{t,1}\calM$ is the unipotent vanishing cycle of $\calM$ along $t=0$,
and we have the morphisms
$\Can \colon \psi_{t,1}\calM\to \phi_{t,1}\calM$ and 
$\Var \colon \phi_{t,1}\calM\to \psi_{t,1}\calM(-1)$,
    where $\psi_{t,1}\calM$ denotes the unipotent part of $\psi_{t}\calM$.
In \cite{SaiMon}, we introduced a category $\calG$ whose objects are
tuples of the form
$(\psi_{t}\calM,\phi_{t,1}\calM,\Can,\Var)$
and proved the following result.
\begin{prop}[Theorem~3.5 of \cite{SaiMon}]
There is an equivalence of categories
\[\MHM_{\mathrm{mon}}(\CC_t)\simto \calG,\]
which sends $\calM$ to $(\psi_{t}\calM,\phi_{t,1}\calM,\Can,\Var)$.
In particular, $\calM$ can be recovered from the tuple.
\end{prop}
Using this description,
one can describe explicitly the Fourier-Laplace transform of monodromic mixed Hodge modules~\cite{SaiMon}.
Such transforms appear in several applications,
for example, in computations of irregular Hodge filtrations~\cite{SaiHodgefilt}, \cite{SabYu23},
the construction of Hodge microsheaves~\cite{KuSa},
and computations of Hodge structures of Fano schemes of intersections of two quadrics~\cite{Belmansetal}.
The explicit description of the Fourier–Laplace transform for monodromic mixed Hodge modules was later generalized to vector bundles of arbitrary rank in \cite{CD23} and \cite{SaiHodgefilt}.
See also \cite{Dirks24} for an application to GKZ systems.

In this paper, we aim to make the above proposition more explicit.
Namely,
we answer the question: what kind of objects does the category $\calG$ consist of?

We focus on the category
$\MHM_{\mathrm{mon},\mathrm{unip}}(\CC_t)$ of monodromic mixed Hodge modules on $\CC_t$ with unipotent monodromy.
For the case without eigenvalue $1$, see the final paragraph of this introduction.
By the above proposition,
the problem is reduced to classifying diagrams of the form
\begin{align}\label{june9-1}
    \xymatrix@C=40pt{\Vd\ar@<0.5ex>[r]^-{f,\ (-1)}&\Ud\ar@<0.5ex>[l]^-{g}},
\end{align}
where $\Vd$ and $\Ud$ are mixed Hodge structures,
$f\colon \Vd\to \Ud(-1)$, $g\colon \Ud\to \Vd$ are morphisms of mixed Hodge structures,
and the compositions
$fg$ and $gf$ are nilpotent.
We denote the corresponding mixed Hodge module by $\calM$.

Let us first consider the underlying perverse sheaves, which are unipotent and monodromic.
Here we say a perverse sheaf on $\CC$ is monodromic
if it has no singular points outside $0$,
and unipotent if the monodromy automorphism is unipotent.
A monodromic unipotent $\taik(=\QQ\ \text{or}\ \CC)$-perverse sheaf is described by a diagram:
\begin{align}\label{june9-2}
    \xymatrix@C=40pt{\Vd\ar@<0.5ex>[r]^-{f}&\Ud\ar@<0.5ex>[l]^-{g}},
\end{align}
where $\Vd$ and $\Ud$ are finite-dimensional $\taik$-vector spaces, $f$ and $g$ are linear maps,
and $fg$ and $gf$ are nilpotent.
The classification of such diagrams is known.
Let $\calL_i$ be a local system on $\CS$
whose monodromy has only one Jordan block of size $i$ with eigenvalue $1$.
For the inclusion $j\colon \CS\to \CC$,
we define four types of perverse sheaves as follows (see Definition~\ref{sep1-1}):
\begin{align*}
    A_i:= j_!\calL_i[1],\quad B_i:= \DR j_*\calL_i[1](=\DR\Gamma_{\CS}(j_{!}\calL_{i}[1])),\\ 
    C_i:=j_{!*}\calL_i[1],\quad D_i:=\cokaa(\taik_{\CC}[1]\to \DR j_*\calL_{i+1}[1]).
\end{align*}
Here $D_i$ is defined for $i\geq 0$,
whereas the other three types are defined for $i\geq 1$,
and $j_{!*}$ denotes the minimal extension functor.
For $i=2$, the diagrams of the form (\ref{june9-2}) corresponding to the perverse sheaves $A_2$, $B_2$, $C_2$ and $D_2$ are as follows.
\begin{center}

\begin{minipage}{0.24\textwidth}
\centering
\parbox[c][3.0cm][c]{\linewidth}{
\[
\xymatrix@C=30pt@R=3pt{
&\bullet\ar[ld]\\
\bullet\ar[dr]&\\
&\bullet\ar[ld]\\
\bullet\ar@{.>}[dr]&\\
&0\\
}
\]
}
\par
{\small Type $A_2$}
\end{minipage}
\hfill
\begin{minipage}{0.24\textwidth}
\centering
\parbox[c][3.0cm][c]{\linewidth}{
\[
\xymatrix@C=30pt@R=3pt{
\bullet\ar[dr]&\\
&\bullet\ar[ld]\\
\bullet\ar[dr]&\\
&\bullet\ar@{.>}[ld]\\
0&\\
}
\]
}
{\small Type $B_2$}
\end{minipage}
\hfill
\begin{minipage}{0.24\textwidth}
\centering
\parbox[c][3.0cm][c]{\linewidth}{
\[
\xymatrix@C=30pt@R=1pt{
&\bullet\ar[ld]\\
\bullet\ar[dr]&\\
&\bullet\ar@{.>}[ld]\\
0&
}
\]
}
{\small Type $C_2$}
\end{minipage}
\hfill
\begin{minipage}{0.24\textwidth}
\centering
\parbox[c][3.0cm][c]{\linewidth}{
\[
\xymatrix@C=30pt@R=1pt{
\bullet\ar[dr]&\\
&\bullet\ar[ld]\\
\bullet\ar[dr]&\\
&\bullet\ar[dl]\\
\bullet\ar@{.>}[dr]&\\
&0
}
\]
}
{\small Type $D_2$}
\end{minipage}
\end{center}
Here, each dot represents $\taik$,
and the direct sums of the copies of $\taik$ in the left and right columns are $V$ and $U$, respectively.
Each solid arrow is the identity map, and each dotted arrow is the zero map.

The following proposition follows from the classification of finite-dimensional nilpotent representations of the cyclic quiver with two vertices.
See Remark~\ref{aug5-11}.
\begin{prop}
Let $K$ be a unipotent and monodromic perverse sheaf on $\CC_t$.
Then, $K$ is isomorphic to a finite direct sum of perverse sheaves of the form $X_i$ for some $X\in \{A,B,C,D\}$.
\end{prop}
Although this result is not new, we give a proof of this proposition in Section~\ref{sectionperv}
because we will later need a similar argument for mixed Hodge structures.

Let $\QQ^H$ be the trivial polarizable pure Hodge structure of weight $0$,
whose Hodge filtration jumps at $0$.
For $i\geq 1$, we define a mixed Hodge structure $K_i$ by
\[K_i:=\QQ^H(i-1)\oplus \QQ^H(i-2)\oplus \cdots \oplus \QQ^H,\]
equipped with the nilpotent morphism
$N_i\colon K_i\to K_i(-1)$ 
whose restriction $\QQ^H(j)\to \QQ^H(j+1)(-1)$ is minus the identity for $0\leq j\leq i-2$,
and whose restriction to $\QQ^H(i-1)$ is zero.
Set $K_{0}=0$.
Let $\calL^H_i$ be the admissible graded-polarizable variation of mixed Hodge structure on $\CS$ whose underlying local system is $\calL_i$
and whose nearby cycle along $t=0$ is $K_i$ with the nilpotent morphism $N_i$.
Its weight filtration is the monodromy weight filtration of $\calL_i$ centered at $1-i$,
and its Hodge filtration is given by the nilpotent orbit determined by $F^\bullet K_i$ and $N_i$.
For an explicit construction of it, see \cite[Proposition~3.9 and Remark~3.10]{SaiMon}.

Taking the pair $(V,U)$ in the diagram (\ref{june9-1}) to be 
\[(K_i,K_i),\quad (K_i(-1),K_i),\quad (K_{i-1},K_i),\quad \mbox{and}\quad (K_{i+1}(-1),K_i)\]
for the types $A_i$, $B_i$, $C_i$ and $D_i$, respectively,
we obtain natural mixed Hodge module structures on these perverse sheaves.
We denote them by 
\begin{align}\label{june9-4}
    A_i^H,\quad B_i^H,\quad C_i^H,\quad D_i^H.
\end{align}
Then, for $i\geq 1$ and $X\in \{A,B,C,D\}$, we have
\[\psi_{t}X^H_i\simeq K_i,\quad \text{and}\quad X^H_i|_{\CS}\simeq \calL^H_i[1].\]

The diagram (\ref{june9-1}) induces another diagram:
\begin{align}\label{june9-3}
    \xymatrix@C=40pt{\GR^W_\bullet \Vd\ar@<0.5ex>[r]^-{\GR^W_\bullet f,\ (-1)}&\GR^W_\bullet \Ud\ar@<0.5ex>[l]^-{\GR^W_\bullet g}},
\end{align}
where $\GR^W_\bullet \Vd$ and $\GR^W_\bullet \Ud$ are 
the associated graded objects with respect to the weight filtrations and 
$\GR^W_\bullet f$ and $\GR^W_\bullet g$ denote the induced morphisms.
We denote by $G(\calM)$ the mixed Hodge module corresponding to (\ref{june9-3}).
\begin{thm}[Theorem~\ref{may15-5}]
There exists a decomposition:
\begin{align}\label{june9-5}
    G(\calM)\simeq \bigoplus_{s=1}^\ell (X^s)^H_{i(s)}\otimes H^s,
\end{align}
where $X^s\in \{A,B,C,D\}$ and 
$i(s)\in \ZZ_{\geq 0}$, with $i(s)\geq 1$ unless $X^s=D$,
$(X^s)^H_{i(s)}$ is one of the mixed Hodge modules in (\ref{june9-4}), and $H^s$ is a pure Hodge structure. 
\end{thm}
Starting from this decomposition of $G(\calM)$,
we study how the original object $\calM$ is obtained.
We call a mixed Hodge module equipped with the chosen decomposition of the form appearing on the right-hand side of (\ref{june9-5}) a split graded model.

It is well known that
extensions of mixed Hodge structures with two graded pieces are classified as follows.
Let $H'$ and $H''$ be pure Hodge structures of weights $w'$ and $w''$,
respectively, with $w'<w''$.
Then mixed Hodge structures $H$ fitting into exact sequences of the form
\[0\to H'\to H\to H''\to 0\]
are classified by
\[\Hom_{\CC}(H'',H')/\left(\Hom_{F}(H'',H')+\Hom_{\QQ}(H'',H')\right),\]
where 
$\Hom_{\CC}(H'',H')$ is the space of $\CC$-linear maps between the underlying $\CC$-vector spaces,
and $\Hom_{F}(H'',H')$ (resp. $\Hom_{\QQ}(H'',H')$) is the subgroup of 
$\Hom_{\CC}(H'',H')$
consisting of maps preserving the Hodge filtration (resp. the $\QQ$-structure).

This description can be generalized to the case of more than two graded pieces.
Let $\Lambda\subset \ZZ$ be a finite subset.
Let $\{L_w\}_{w\in \Lambda}$ be a family of pure Hodge structures,
where $L_w$ is of weight $w$.
We denote by $L_{w,\CC}$ the underlying $\CC$-vector space and by $L_{w,\CC}^{p,q}$ the $(p,q)$-part of the Hodge decomposition of $L_w$ for $p,q\in \ZZ$ with $p+q=w$.
We denote by $\Exten{\{L_w\}_{w\in \Lambda}}$ the set of isomorphism classes
of pairs $(H,\{\alpha_w\}_{w\in \Lambda})$,
where $H$ is a mixed Hodge structure and $\GR^W_wH=0$ for $w\notin \Lambda$,
and $\alpha_w\colon L_w\simto \GR^W_{w}H$ is an isomorphism of pure Hodge structures.
We define $\UpT:=\UpT(\{L_w\}_{w\in \Lambda})$ to be the group of matrices
 $(m_{ij})_{i,j\in \Lambda}$
such that
$m_{ij}\in \Hom_{\CC}(L_{j,\CC},L_{i,\CC})$,
$m_{ii}=\iden$ for $i\in \Lambda$ 
and 
$m_{ij}=0$ for $i>j$.
Thus matrices in $\UpT$ are upper triangular with respect to the usual order on $\Lambda$.
We also define subgroups $\UpT_F$ and $\UpT_\QQ$
consisting of matrices whose components preserve the Hodge filtration and the underlying $\QQ$-structure, respectively.

\begin{prop}[Proposition~\ref{jun16-2} and Lemma~\ref{may7-1}]\label{june9-8}
There is a bijection:
    \[\Exten{\{L_w\}_{w\in \Lambda}}\simeq \UpT_\QQ\backslash \UpT/ \UpT_F.\]
Moreover, for a class in $\UpT_\QQ\backslash \UpT/ \UpT_F$,
we can always choose a representative $M=(m_{ij})\in \UpT$ satisfying the following condition:
\[m_{ij}(L_{j,\CC}^{p,q})\subset \bigoplus_{\substack{p'<p\\p'+q'=i}}L_{i,\CC}^{p',q'},\]
which we will call condition $(*)$. 
\end{prop}

For $M=(m_{ij})\in \UpT$,
the mixed Hodge structure $H_M$ corresponding to the class of $M$ under the bijection in Proposition~\ref{june9-8} is constructed as follows.
The underlying $\CC$-vector space $H_{M,\CC}$,
the $\QQ$-vector space $H_{M,\QQ}$,
and the weight filtration $W_\bullet H_{M,\QQ}$ are the same as those of the mixed Hodge structure $\bigoplus_{w\in \Lambda}L_w$,
obtained as the direct sum of pure Hodge structures.
The Hodge filtration of $H_M$ is defined by
\begin{align}\label{june10-1}
    F^pH_{M,\CC}=M\left(\bigoplus_{w\in \Lambda}F^pL_{w,\CC}\right),
\end{align}
where
$M$ is regarded as a $\CC$-linear automorphism of $\bigoplus_{w\in \Lambda}L_{w,\CC}$.

The $\Vd$-part and the $\Ud$-part of the diagram corresponding to 
the right-hand side of (\ref{june9-5})
are of the form:
\begin{align}
    &\bigoplus_{s=1}^\ell\bigoplus_{j\in J_\Vd(s)}H^s(j)\quad \text{and}\label{june9-6}\\
    &\bigoplus_{s=1}^\ell \bigoplus_{j\in J_{\Ud}(s)}H^s(j),\label{june9-7}
\end{align}
respectively, where $J_\Vd(s)$ and $J_{\Ud}(s)$ are certain finite subsets of $\ZZ$.
We define $L_{\Vd,w}$ (resp. $L_{\Ud,w}$)
to be the direct sum of the summands $H^s(j)$ in (\ref{june9-6}) (resp. (\ref{june9-7}))
whose weight is $w$.
Then, by Proposition~\ref{june9-8},
the original $\Vd$ and $\Ud$ can be described by
upper triangular matrices $M_\Vd\in \UpT_\Vd=\UpT(\{L_{\Vd,w}\}_{w})$ and $M_\Ud\in \UpT_\Ud=\UpT(\{L_{\Ud,w}\}_{w})$.
Thus, we obtain a triple
\[(G(\calM),M_\Vd,M_\Ud).\]

By extracting the properties satisfied by such triples,
we define a category $\calC$.
Its objects are triples $(G,M_\Vd,M_\Ud)$, where $G$ is a split graded model,
$M_\Vd$ and $M_\Ud$ are matrices associated with the $\Vd$-part and the $\Ud$-part of the diagram of $G$, respectively,
satisfying compatibility conditions with the morphisms between the $\Vd$-part and the $\Ud$-part; see Definition~\ref{may18-5}.

\begin{prop}[Proposition~\ref{aug4-1-prop}]
The category $\calC$ is equivalent to $\MHM_{\mathrm{mon},\mathrm{unip}}(\CC_t)$.
\end{prop}

We then use this description to determine the isomorphism classes of objects in 
$\MHM_{\mathrm{mon},\mathrm{unip}}(\CC_t)$.
The point is that the category $\calC$ contains redundant data.
Since morphisms of mixed Hodge structures are strict,
the Hodge filtration on some parts is determined by the Hodge filtration on other parts.
For example, the Hodge filtration on the image of $f$ is determined by that of $\Vd$.
Moreover, some parts of $M_\Vd$ and $M_\Ud$ must vanish.
Thus the data of $M_\Vd$ and $M_\Ud$ are in fact determined by a much smaller set of components.

Let $G=\bigoplus_{s=1}^\ell (X^s)^H_{i(s)}\otimes H^s$ be a split graded model.
For $1\leq s \leq \ell$, we define $Z_s\in \{\Vd,\Ud\}$ to be the side from which the diagram of $(X^s)_{i(s)}^H$ starts (see (\ref{aug5-9}))
and define $a_s$ as the integer such that the first term of the diagram of $(X^s)^H_{i(s)}$ is $\QQ^H(a_s)$ (see (\ref{oct8-1})).
We define $\Delta$ as a set of indices
$(s,t,q)\in ([1,\ell]\cap \ZZ)^2\times \ZZ$
satisfying the following conditions (see Definition~\ref{may22-1} for the precise definition):
\begin{enumerate}
    \item $H^t(q)$ appears in the $Z_s$-part of the diagram of $G$, that is, in (\ref{june9-6}) or (\ref{june9-7}).
    \item The weight of $H^s(a_s)$ is greater than that of $H^t(q)$.
    \item After applying $f$ and $g$ successively, the sequence starting from $H^s(a_s)$ does not vanish earlier than the sequence starting from $H^t(q)$.
\end{enumerate}
Then, we have the following.

\begin{thm}[Theorem~\ref{aug5-5}]
Let $\frC_3$ be the set of pairs $(G,\{\mu_{s,t,q}\}_{(s,t,q)\in \Delta})$,
where $G$ is a split graded model and $\{\mu_{s,t,q}\}_{(s,t,q)\in \Delta}$ is a family of $\CC$-linear maps
\[\mu_{s,t,q}\in \Hom_{\CC}(H^s(a_s),H^t(q)).\]
Then,
we have a surjective map
\[\frC_3\twoheadrightarrow \MHM_{\mathrm{mon},\mathrm{unip}}(\CC_t)/\simeq,\]
where the right-hand side is the set of isomorphism classes of objects in $\MHM_{\mathrm{mon},\mathrm{unip}}(\CC_t)$.
\end{thm}

The above surjection is explicit.
Given $(G,\{\mu_{s,t,q}\}_{(s,t,q)\in \Delta})\in \frC_3$,
let $\Vd_G$ and $\Ud_G$ be the $\Vd$-part and the $\Ud$-part of $G$ as mixed Hodge structures.
Starting from the data $\{\mu_{s,t,q}\}_{(s,t,q)\in\Delta}$, 
we complete these maps to matrices $M_\Vd$ and $M_\Ud$
so that $(G,M_\Vd,M_\Ud)$ is in $\calC$.
By the formula (\ref{june10-1}),
these matrices define new Hodge filtrations on the underlying vector spaces of $\Vd_G$ and $\Ud_G$,
and hence new mixed Hodge structures $\Vd_G^{M_\Vd}$ and $\Ud_{G}^{M_\Ud}$,
together with a diagram of the form (\ref{june9-1}).
This means that every object of $\MHM_{\mathrm{mon},\mathrm{unip}}(\CC_t)$ 
can be represented by a split graded model together with a family of maps $\{\mu_{s,t,q}\}_{(s,t,q)\in \Delta}$.

Let us give an example.
Consider $G=D_2^H(1)\oplus A_1^H$, where $D_2^H$ and $A_1^H$ are as in (\ref{june9-4}).
We set $(X^1)^H_{i(1)}=D_2^H$, $H^1=\QQ^H(1)$, $(X^2)^H_{i(2)}=A_1^H$
and 
$H^2=\QQ^H$.
In this case, the diagram of $G$ is as follows:
\[
\xymatrix@C=25pt@R=3pt{
\QQ^H\ar[rd]|-{\scriptscriptstyle (-1)}&&&\QQ^H\ar[ld]\\
&\QQ^H(1)\ar[ld]&\QQ^H\ar@{.>}[dr]&\\
\QQ^H(1)\ar[dr]|-{\scriptscriptstyle (-1)}&\ar@{}[r]|{\oplus}&&0\\
&\QQ^H(2)\ar[ld]&&\\
\QQ^H(2)\ar@{.>}[dr]&&&\\
&0&&,\\
}
\]
where $\QQ^H$ is the trivial Hodge structure of weight $0$.
We have $\Vd_G=(\QQ^H\oplus \QQ^H(1)\oplus \QQ^H(2))\oplus \QQ^H$ and $\Ud_G=(\QQ^H(1)\oplus \QQ^H(2))\oplus \QQ^H$.
Then, 
$\{\mu_{s,t,q}\}$ consists of three maps $(\mu_{1,1,0},\mu_{1,1,1},\mu_{2,1,1})$:
\[
\xymatrix@C=25pt@R=3pt{
\QQ^H\ar[rd]|-{\scriptscriptstyle (-1)}
\ar@/_10pt/[dd]^{\mu_{1,1,0}}\ar@/_28pt/[dddd]_{\mu_{1,1,1}}\ar@/^18pt/[drr]^{\xcancel{\mu_{1,2,0}}}&&&\QQ^H\ar[ld]\ar@/^18pt/[dll]^{\xcancel{\mu_{2,1,0}}}\ar@/^18pt/[dddll]^{\mu_{2,1,1}}\\
&\QQ^H(1)\ar[ld]&\QQ^H\ar@{.>}[dr]&\\
\QQ^H(1)\ar[dr]|-{\scriptscriptstyle (-1)}&\ar@{}[r]|{\oplus}&&0\\
&\QQ^H(2)\ar\ar[ld]&&\\
\QQ^H(2)\ar@{.>}[dr]&&&\\
&0&&.\\
}
\]
Here,
$\mu_{1,2,0}$ (resp. $\mu_{2,1,0}$) is absent
because the second (resp. third) condition of $\Delta$ is not satisfied.
This means that
any object $\calM\in \MHM_{\mathrm{mon},\mathrm{unip}}(\CC_t)$
such that $G(\calM)$ is isomorphic to $D_2^H(1)\oplus A_1^H$
can be represented by $(\mu_{1,1,0},\mu_{1,1,1},\mu_{2,1,1})$.

Finally, we consider the case where the monodromy does not have eigenvalue $1$.
In this case,
every monodromic perverse sheaf is the zero extension of its restriction to $\CS$,
so the classification is simpler (see Proposition~\ref{june14-1}).
This result can be reinterpreted using the classification in the unipotent case,
as explained in Subsection~\ref{aug5-2}.
Using this, we obtain analogues of the classification results for the unipotent case in Subsection~\ref{ssclasmhmnonunip}.

\subsection*{Acknowledgments}
I would like to thank Claude Sabbah for pointing out the reference \cite{Monvel} to me (see Remark~\ref{aug5-11}) and for reading a draft of this paper and giving me helpful comments.
I would also like to thank Tatsuki Kuwagaki and Yota Shamoto for their encouragement.
This work was supported by JSPS KAKENHI Grant Number JP25K17236.

\subsection*{Use of AI}
The author used ChatGPT (OpenAI, on the \$22/month plan) to assist with improving the English presentation,
organizing and improving some proofs,
and checking notation and consistency.
Although parts of the wording suggested by AI were adopted,
the final manuscript was prepared and typed by the author, who takes full responsibility for its contents.
All proofs were independently checked by the author without the use of AI.

\section{Monodromic mixed Hodge modules on $\CC$}

\subsection{Diagram description}
Let $\CC_t$ be an affine line with coordinate $t$
and let $j\colon \CS:=\CC_t\setminus\{0\}\hookrightarrow \CC_t$ denote the inclusion.
We denote by $\MHM(\CC_t)$ the category of 
algebraic graded-polarizable mixed Hodge modules on $\CC_t$.
As mentioned in the Introduction,
we will omit the adjectives ``algebraic" and ``graded-polarizable".
An object of $\MHM(\CC_t)$ is represented by a tuple
$(M,F^\bullet M,K,W_\bullet K)$,
where $M$ is the underlying $D$-module, $F^\bullet M$ is the Hodge filtration,
$K$ is the underlying $\QQ$-perverse sheaf
with an isomorphism $\mathrm{DR}(M)\simto \CC\otimes_\QQ K$,
where $\mathrm{DR}$ is the de Rham functor,
and
$W_\bullet K$ is the weight filtration.

\begin{defi}
    For a perverse sheaf $K$ on $\CC_t$ with coefficients in a field $\taik$,
we say that $K$ is monodromic if $K|_{\CS}$ is a shifted local system.
We say that $K$ is unipotent if there exists a finite subset 
$S\subset \CC$ such that $K|_{\CC\setminus S}$ is a shifted local system
and the local monodromy at each point of $S$ is unipotent.
We denote by $\Perv_{\mon}(\taik_{\CC_t})$ (resp. $\Perv_{\mon,\unip}(\taik_{\CC_t})$)
the category of monodromic (resp. unipotent and monodromic) perverse sheaves.
    For a mixed Hodge module $\calM=(M,F^\bullet M,K,W_\bullet K)$ on $\CC_t$,
    if $K$ is monodromic (resp. unipotent), then we say that $\calM$ is monodromic (resp. unipotent).
    We denote by $\MHM_{\mathrm{mon}}(\CC_t)$ (resp. $\MHM_{\mathrm{mon},\mathrm{unip}}(\CC_t)$) the category of monodromic (resp. unipotent and monodromic) mixed Hodge modules on $\CC_t$.
\end{defi}

\begin{rem}
    Another characterization of monodromic objects
    in terms of $D$-modules is explained in \cite{SaiMon}.
\end{rem}

For a monodromic perverse sheaf $K\in \Perv_{\mon}(\taik_{\CC_t})$,
there is a canonical decomposition $K=K_1\oplus K_{\neq 1}$
such that $K_1\in \Perv_{\mon,\unip}(\taik_{\CC_t})$
and $K_{\neq 1}$ is the zero-extension $j_!\calL[1]$ of some local system $\calL$ on $\CS$ whose monodromy does not have eigenvalue $1$.
Moreover,
the category $\Perv_{\mon,\unip}(\taik_{\CC_t})$ is equivalent to the category of diagrams of finite-dimensional $\taik$-vector spaces:
\begin{align}\label{june11-6}
    \xymatrix@C=40pt{\Vd\ar@<0.5ex>[r]^-{v}&\Ud\ar@<0.5ex>[l]^-{c}},
\end{align}
such that $vc$ is nilpotent.
More concretely,
an object $K$ corresponds to the diagram
\begin{align}\label{june11-4}
    \xymatrix@C=50pt{{}^p\phi_{t}K\ar@<0.5ex>[r]^-{\chiisaivar}&{}^p\psi_{t}K\ar@<0.5ex>[l]^-{\Can}},
\end{align}
where ${}^p\psi_{t}$ (resp. ${}^p\phi_{t}$) denotes the nearby cycle (resp. vanishing cycle) functor of the perverse sheaves, 
and $\Can$ and $\chiisaivar$
denote the natural morphisms associated with these functors.
This description admits a Hodge-theoretic enhancement.

%since $\Var$ is the morphism from ${}^p\phi_{t}K$ to ${}^p\psi_{t}K(-1)=\frac{1}{2\pi\sqrt{-1}}{}^p\psi_{t}K$,
%we need to multiply $\Var$ by $2\pi\sqrt{-1}$ to obtain a morphism ${}^p\phi_{t}K\to {}^p\psi_{t}K$.
%Therefore, Corollary~\ref{june11-3} is a Hodge theoretic refinement of this equivalence.

In \cite{SaiMon}, we introduced the category $\calG(X)$ for a complex manifold $X$.
In this paper, we take $X=\pt$ and write $\calG=\calG(\pt)$.
An object of $\calG$ is a tuple $((\Ud,T_s,N),\Vd,c,v)$,
where $\Ud$ and $\Vd$ are mixed Hodge structures and
\begin{align*}
    T_s\colon \Ud&\simto \Ud\\
    N\colon \Ud&\to \Ud(-1)\\
    c\colon \Ud_0(:=&\kaa(T_s-1))\to \Vd\\
    v\colon \Vd&\to \Ud_0(-1)
\end{align*}
are morphisms of mixed Hodge structures with the following properties:
\begin{enumerate}
    \item $T_s$ commutes with $N$ and $N$ is nilpotent.
    \item The underlying $\CC$-vector space $\Ud_\CC$ of $\Ud$ is decomposed as
    \[\Ud_\CC=\bigoplus_{\alpha\in (-1,0]\cap \QQ}\Ud_{\CC,\alpha},\]
    where $\Ud_{\CC,\alpha}:=\kaa(T_s-\exp(-2\pi\sqrt{-1}\alpha))\subset \Ud_{\CC}$.
    \item $vc\colon \Ud_0\to \Ud_0(-1)$ is $-N$. 
\end{enumerate}
Then, for a mixed Hodge module $\calM\in \MHM_{\mathrm{mon}}(\CC_t)$ on $\CC_t$,
the following tuple defines an object in $\calG$:
\begin{align}\label{june11-1}
    \left(\left(\psi_{t}\calM,T_s,N\left(=-\frac{1}{2\pi\sqrt{-1}}\log T_u\right)\right),\phi_{t,1}\calM,\Can,\Var\right),
\end{align}
where 
$\psi_{t}\calM$ (resp. $\phi_{t,1}\calM$) is the nearby cycle (resp. the unipotent part of the vanishing cycle) of $\calM$,
and $T_s$ (resp. $T_u$) is the semisimple (resp. unipotent) part of the monodromy automorphism of $\psi_{t}\calM$.

\begin{prop}[Theorem~3.5 of \cite{SaiMon}]\label{june11-5}
    We have an equivalence of categories:
    \begin{align}
        \MHM_{\mathrm{mon}}(\CC_t)\simto \calG,
    \end{align}
    which sends $\calM\in  \MHM_{\mathrm{mon}}(\CC_t)$ to 
    $((\psi_{t}\calM,T_s,N),\phi_{t,1}\calM,\Can,\Var)$.
\end{prop}

Moreover, we define $\calG_{\unip}$ to be a subcategory of $\calG$
consisting of tuples $((\Ud,T_s,N),\Vd,c,v)$ such that $T_s$ is the identity.
Then, giving an object of $\calG_{\unip}$ is 
equivalent to giving a diagram of mixed Hodge structures:
\begin{align}\label{june11-2-1}
    \xymatrix@C=40pt{\Vd\ar@<0.5ex>[r]^-{v,\ (-1)}&\Ud\ar@<0.5ex>[l]^-{c}},
\end{align}
such that $vc$, and hence also $cv$, is nilpotent. 

\begin{cor}\label{june11-3}
    We have an equivalence of categories:
    \begin{align}
        \MHM_{\mathrm{mon},\unip}(\CC_t)\simto \calG_\unip,
    \end{align}
    which sends $\calM\in  \MHM_{\mathrm{mon},\unip}(\CC_t)$ to the diagram
    \begin{align}\label{june11-2-2}
    \xymatrix@C=40pt{\phi_{t}\calM(=\phi_{t,1}\calM)\ar@<0.5ex>[r]^-{\Var,\ (-1)}&\psi_{t}\calM(=\psi_{t,1}\calM)\ar@<0.5ex>[l]^-{\Can}}.
\end{align}
\end{cor}

\subsection{The unipotent part and the part without eigenvalue $1$}\label{aug5-1}

In this paper, 
we aim to classify the category $\MHM_{\mathrm{mon}}(\CC_t)$.
We denote by $\MHM_{\mathrm{mon},\neq 1}(\CC_t)$ the full subcategory of $\MHM_{\mathrm{mon}}(\CC_t)$ consisting of objects $\calM$ such that
$\psi_{t,1}\calM=0$ and $\phi_{t,1}\calM=0$.
As in the case of monodromic perverse sheaves,
every object $\calM\in \MHM_{\mathrm{mon}}(\CC_t)$
admits a canonical decomposition $\calM=\calM'\oplus \calM''$,
where $\calM'\in \MHM_{\mathrm{mon},\unip}(\CC_t)$
and $\calM''\in \MHM_{\mathrm{mon},\neq 1}(\CC_t)$.
The summand $\calM''$ is the zero-extension $j_!(\calM''|_{\CS})$ of its restriction $\calM''|_{\CS}$.
Therefore, by Proposition~\ref{june11-5},
classifying such objects $\calM''$
amounts to classifying tuples 
$(\Ud,T_s,N)$,
where $\Ud$ is a mixed Hodge structure, 
$T_s\colon \Ud\simto \Ud$ is an automorphism satisfying $T_s^m=\iden$ for some $m\in \ZZ_{\geq 1}$ and having no eigenvalue $1$,
and $N\colon \Ud\to \Ud(-1)$ is a nilpotent morphism commuting with $T_s$.
After forgetting $T_s$,
the pair $(\Ud,N)$ can be regarded as a special case of a diagram of the form~(\ref{june11-2-1}),
with $V=U$, $v=-N$ and $c=\iden$.
Thus,
the classification of $\calM''$ 
reduces to the classification of diagrams of the form (\ref{june11-2-1})
together with an automorphism $T_s\colon \Ud\simto \Ud$ compatible with the diagram.

For this reason, 
the key step is the classification of the unipotent case.
We proceed as follows.
In Subsection~\ref{ssunipperv},
we classify the objects of 
$\Perv_{\mon,\unip}(\taik_{\CC_t})$,
and then use this to treat monodromic perverse sheaves without eigenvalue $1$ in Subsection~\ref{aug5-2}.
Next, in Subsection~\ref{ssclasmonmhm},
we classify the objects of 
$\MHM_{\mathrm{mon},\unip}(\CC_t)$,
and then use this to treat monodromic mixed Hodge modules without eigenvalue $1$ in Subsection~\ref{ssclasmhmnonunip}.

%\begin{align}\label{june11-4}
 %   \xymatrix@C=30pt{\Ud\ar@<0.5ex>[r]^-{N,\ (-1)}&\Ud\ar@<0.5ex>[l]^-{=}},
%\end{align}

%Hence, the classification of $\MHM_{\mathrm{mon},\unip}(\CC_t)$, or equivalently diagrams of the form (\ref{june11-2}), is the key point.

%where $W$ is a mixed Hodge structure, $T_s\colon W\simto W$ is an automorphism on $W$
%which induces a semisimple map on the underlying $\CC$-vector space of $W$,
%and $N\colon W\to W(-1)$ is a nilpotent endomorphism which commutes with $T_s$.

\section{Classification of monodromic perverse sheaves on $\CC$}
\label{sectionperv}

\subsection{The unipotent case}\label{ssunipperv}

Let $\mathbf{k}=\QQ$ or $\CC$.
In this section,
we recall the classification of diagrams of the form:
\begin{align}\label{jan1-1}
    \xymatrix{\Vd\ar@<0.5ex>[r]^-{f}&\Ud\ar@<0.5ex>[l]^-{g}},
\end{align}
where $\Vd$ and $\Ud$ are finite-dimensional $\mathbf{k}$-vector spaces,
$f\colon \Vd\to \Ud$ and $g\colon \Ud\to \Vd$ are linear maps between vector spaces such that 
    $gf$ and $fg$ are nilpotent.
We denote $gf$ (resp. $fg$) by $N_\Vd$ (resp. $N_\Ud$).
As we mentioned in the previous section,
the category of such diagrams is equivalent to 
$\Perv_{\mon,\unip}(\taik_{\CC_t})$.

\begin{comment}
    
\begin{rem}\say{fieldで大丈夫か？どうするか？tate twistあるから本当はCかQであるべきだが・・・いや、これHMのときだけ(-1)が出るのかな。→そうだね。alg closedいらんかな}
    For a local system $\calL$ on $\CS_t$ whose monodromy automorphism $T$ does not have an eigenvalue $1$,
    there is only one extension of it to a perverse sheaf on $\CC_t$, i.e. $j_!\calL$ for $j\colon \CS\to \CC$.
    Therefore, the category of monodromic perverse sheaves whose monodromy automorphism do not have the eigenvalue $1$ has been already classified,
    because $\calL$ is determined only by the monodromy automorphism.
    In this case,
    the nearby-vanishing cycle diagram\say{この用語あとできめる} of $j_!\calL$ is 
    of the form:\say{これ等価だが間違いだね。varを何にするかは検討すべき。固有値の情報失ってるからかな}
    \begin{align}\label{jan1-2}
    \xymatrix{W(=\phi_{t,\neq 1}j_!\calL)\ar@<0.5ex>[r]^-{f}&W(=\psi_{t,\neq 1}j_!\calL)\ar@<0.5ex>[l]^-{=}},
\end{align}
where $f$ is the nilpotent part of the monodromy of $\calL$.
Therefore, we can also say that 
the category of monodromic perverse sheaves whose monodromy automorphisms do not have the eignvalue $1$ can be expressed as the above diagram with the semi-simple automorphism of $W$ without eigenvalue $1$ which commutes with $f$.
\end{rem}
\end{comment}

\begin{rem}\label{aug5-11}
More generally, the classification of finite-dimensional nilpotent representations of a cyclic quiver is known; see, for example \cite[Proposition~3.26]{Schiffmann}  (Proposition~3.24 in the arXiv version).
The classification of diagrams of the form (\ref{jan1-1}) is a special case of
the classification of representations of a cyclic quiver, namely the case with two vertices.
For $\taik=\CC$,
the classification of (\ref{jan1-1}) is equivalent to the classification of
algebraic regular holonomic $D$-modules on $\CC$
whose only possible singular point is
$0$ and whose monodromy is unipotent.
The classification of such $D$-modules is described in \cite[pp. 316–317]{Monvel}.
However, 
since the constructions of subquotients and splittings of $\Vd$ and $\Ud$ used in the classification of (\ref{jan1-1}) will be needed later in the
Hodge-theoretic version in Section~\ref{sectionmhm},
we record the proof in this section with the necessary details,
even though the arguments are purely linear algebraic and not new.
\end{rem}

We aim to find a normal form for diagrams of the form (\ref{jan1-1}).
To this end, we introduce some subquotients of $\Vd$ and $\Ud$.
%and then we construct bases of $\Vd$ and $\Ud$ by transporting basis elements of the subquotients along $f$, $g$, and their compositions.
 
\begin{defi}\label{mar4-3}
We define subquotients of $\Vd$ and $\Ud$ as follows.
\begin{align*}
    \Vd_i&:=\Ker{fN_\Vd^{i}}/\left(\Ker{N_\Vd^i}+\left(\Ima{g}\cap \Ker{fN_\Vd^i}\right)\right),\\
    \Vd_i'&:=\Ker{N_\Vd^{i}}/\left(\Ker{fN_\Vd^{i-1}}+\left(\Ima{g}\cap \Ker{N_\Vd^i}\right)\right),\\
\Ud_i&:=\Ker{gN_\Ud^{i}}/\left(\Ker{N_\Ud^i}+\left(\Ima{f}\cap \Ker{gN_\Ud^i}\right)\right),\quad \mbox{and}\\
    \Ud_i'&:=\Ker{N_\Ud^{i}}/\left(\Ker{gN_\Ud^{i-1}}+\left(\Ima{f}\cap \Ker{N_\Ud^i}\right)\right),
\end{align*}    
where $N_\Vd^0=\Id_\Vd$ and $N_\Ud^0=\Id_\Ud$.
Here $\Vd_{i}$ and $\Ud_{i}$ are defined for $i\geq 0$, and $\Vd'_{i}$ and $\Ud'_{i}$ are defined for $i\geq 1$.
We also set
\begin{align*}
    \VT{i}:=\Ker{fN_\Vd^{i}}/\Ker{N_\Vd^i},\quad 
        \VTp{i}:=\Ker{N_\Vd^{i}}/\Ker{fN_\Vd^{i-1}},\\
    \UT{i}:=\Ker{gN_\Ud^{i}}/\Ker{N_\Ud^i},\quad
        \UTp{i}:=\Ker{N_\Ud^{i}}/\Ker{gN_\Ud^{i-1}}.
\end{align*}
Similarly, $\VT{i}$ and $\UT{i}$ are defined for $i\geq 0$, and $\VTp{i}$ and $\UTp{i}$ are defined for $i\geq 1$.

\end{defi}

Note that we have $fN_\Vd^i=N_\Ud^if$ and $gN_\Ud^i=N_\Vd^ig$.
The following lemma is then straightforward.
\begin{lem}\label{jan7-1}
\begin{enumerate}
\item The map $f$ (resp. $g$) induces
maps $\VT{i}\to \UTp{i}$ and $\VTp{i}\to \UT{i-1}$
(resp. $\UT{i}\to \VTp{i}$ and $\UTp{i}\to \VT{i-1}$),
which we denote by the same symbol $f$ (resp. $g$).

    \item 
    We obtain the following sequences of inclusions between subquotients of $\Vd$ and $\Ud$:
    \begin{align*}
        \dots \xhookrightarrow{\ g\ }
\VT{i}        \xhookrightarrow{\ f\ }
\UTp{i}\xhookrightarrow{\ g\ }
\VT{i-1}\xhookrightarrow{\ f\ }
         &\cdots \xhookrightarrow{\ f\ }\UTp{2}
        \xhookrightarrow{\ g\ }\VT{1}
    \xhookrightarrow{\ f\ }\UTp{1}\xhookrightarrow{\ g\ }\VT{0},\quad \mbox{and}\\
       \dots \xhookrightarrow{\ f\ }
 \UT{i}\xhookrightarrow{\ g\ }
        \VTp{i}\xhookrightarrow{\ f\ }
        \UT{i-1}\xhookrightarrow{\ g\ }&
        \cdots \xhookrightarrow{\ g\ }\VTp{2}
        \xhookrightarrow{\ f\ }\UT{1}
        \xhookrightarrow{\ g\ }\VTp{1}\xhookrightarrow{\ f\ }\UT{0}.
    \end{align*}
    \end{enumerate}  
    \end{lem}

  %\begin{align*}
   %     \dots \xhookrightarrow{g}
    %    \Ker{fN_{\Vd}^{i}}/\Ker{N_{\Vd}^{i}}&\xhookrightarrow{f}
     %   \Ker{N_\Ud^{i}}/\Ker{gN_\Ud^{i-1}}\xhookrightarrow{g}
      %  \Ker{fN_{\Vd}^{i-1}}/\Ker{N_{\Vd}^{i-1}}\xhookrightarrow{f}
 %       \dots \\
  %      \dots \xhookrightarrow{f}&\Ker{N_\Ud^2}/\Ker{gN_\Ud}
  %      \xhookrightarrow{g}\Ker{fN_{\Vd}}/\Ker{N_{\Vd}}
  %  \xhookrightarrow{f}\Ker{N_\Ud}/\Ker{g}\xhookrightarrow{g}\Ker{f},\\
  %      \dots \xhookrightarrow{f}
  %      \Ker{gN_\Ud^{i}}/\Ker{N_\Ud^{i}}&\xhookrightarrow{g}
  %%      \Ker{N_{\Vd}^{i}}/\Ker{fN_\Ud^{i-1}}\xhookrightarrow{f}
   %     \Ker{gN_\Ud^{i-1}}/\Ker{N_\Ud^{i-1}}\xhookrightarrow{g}
   %     \dots \\
   %     \dots \xhookrightarrow{g}&\Ker{N_{\Vd}^2}/\Ker{fN_{\Vd}}
   %     \xhookrightarrow{f}\Ker{gN_\Ud}/\Ker{N_\Ud}
   %     \xhookrightarrow{g}\Ker{N_{\Vd}}/\Ker{f}\xhookrightarrow{f}\Ker{g}.
   % \end{align*}

   % Since we can regard $\Vd_i$ (resp. $\Vd_i'$, $\Ud_i$ and $\Ud_i'$) as a subspace of $\VT{i}$ (resp. $\VTp{i}$, 
%$\UT{i}$ and $\UTp{i}$),
%we will consider $\Vd_i$, $\Vd_i'$, $\Ud_i$ and $\Ud_i'$ are parts of the sequences above.

In what follows, by choosing several splittings,
we will non-canonically embed $\Vd_i$, $\Vd_i'$, $\Ud_i$, $\Ud_i'$, and their tilde variants
into $\Vd$ and $\Ud$.

\begin{lem}\label{jan7-2}
    We obtain exact sequences as follows.
    \begin{align*}
        0     \longrightarrow \UTp{i+1}\xhookrightarrow{\ g\ }
        &\VT{i}
        \longrightarrow \Vd_{i}     \longrightarrow 0,\\
        0\lra \UT{i}\xhookrightarrow{\ g\ }
        &\VTp{i}\lra \Vd'_i\lra 0,\\
        0\lra \VTp{i+1}\xhookrightarrow{\ f\ }
        &\UT{i}\lra \Ud_i\lra 0,\\
        0\lra \VT{i}\xhookrightarrow{\ f\ }
        &\UTp{i}\lra \Ud_i'\lra 0.
    \end{align*}
    In particular,
    we have non-canonical decompositions:
    \begin{align}\label{mar4-6}
        \VT{i}\simeq g(\UTp{i+1})\oplus \Vd_i,\ 
        \VTp{i}\simeq g(\UT{i})\oplus \Vd_i',\ 
        \UT{i}\simeq f(\VTp{i+1})\oplus \Ud_i\mbox{,\  and\quad }
        \UTp{i}\simeq f(\VT{i})\oplus \Ud_i'.
    \end{align}
    
\end{lem}

\begin{proof}
    If an element $u\in \Ud$ satisfies $g(u)\in \Ima{g}\cap \Ker{fN_\Vd^i}$,
then $u$ is in $\Ker{N_\Ud^{i+1}}$.
Then, combining this observation with Lemma~\ref{jan7-1} yields the first exact sequence.
The others follow similarly.
\end{proof}

%We take $i_0\geq 1$ such that $N_{\Vd}^{i_0}=0$ and $N_\Ud^{i_0}=0$.
%In particular, we have $V_{i_0}=0$, $V_{i_0+1}'=0$, $W_{i_0}=0$ and $W_{i_0+1}'=0$.

\begin{lem}\label{mar4-8}
Consider the following commutative diagram with exact rows for $i\in \ZZ_{\geq 1}$:
\begin{align*}
     \xymatrix{
0\ar[r]&\Ker{N^i_\Vd}\ar[d]^f\ar[r]&\Ker{fN_\Vd^i}\ar[r]\ar[d]^f&\wt{\Vd_i}\ar[d]^f\ar[r]&0\\
0\ar[r]&\Ker{gN_\Ud^{i-1}}\ar[d]^g\ar[r]&\Ker{N^i_\Ud}\ar[r]\ar[d]^g&\wt{\Ud_i'}\ar[d]^g\ar[r]&0\\
0\ar[r]&\Ker{N^{i-1}_\Vd}\ar[r]&\Ker{fN_\Vd^{i-1}}\ar[r]&\wt{\Vd_{i-1}}\ar[r]&0.
        }
\end{align*}
There exist sections
$\wt{\Vd_i}\to \Ker{fN^i}$
(resp. $\wt{\Ud_i'}\to \Ker{N^i_\Ud}$)
of $\Ker{fN^i_\Vd}\twoheadrightarrow \wt{\Vd_i}$
(resp.  $\Ker{N_\Ud^i}\twoheadrightarrow \wt{\Ud_i'}$)
for each $i$ such that
the above diagram obtained by adding these sections is commutative.
The same statement holds for the diagram:
\begin{align*}
     \xymatrix{
0\ar[r]&\Ker{N^i_\Ud}\ar[d]\ar[r]&\Ker{gN_\Ud^i}\ar[r]\ar[d]^g&\wt{\Ud_i}\ar[d]^g\ar[r]&0\\  
0\ar[r]&\Ker{fN_\Vd^{i-1}}\ar[d]\ar[r]&\Ker{N^i_\Vd}\ar[r]\ar[d]^f&\wt{\Vd_i'}\ar[d]^f\ar[r]&0\\
0\ar[r]&\Ker{N^{i-1}_\Ud}\ar[r]&\Ker{gN_\Ud^{i-1}}\ar[r]&\wt{\Ud_{i-1}}\ar[r]&0.
        }
\end{align*}

In particular, there exist decompositions of $\Ker{N_\Vd^i}$,
$\Ker{fN_\Vd^i}$, $\Ker{N_\Ud^i}$ and $\Ker{gN_\Ud^i}$ for each $i$
such that the following diagrams commute

\begin{equation}\label{ap1-2}
    \begin{aligned}
\xymatrix{    \Ker{fN_\Vd^i}\ar@{}[r]|{\simeq }\ar[d]^{f}& \Ker{N_\Vd^i}\ar@{}[r]|-{\oplus }\ar[d]^{f}&\VT{i}\ar[d]^{f}\\
\Ker{N_\Ud^i}\ar@{}[r]|{\simeq }\ar[d]^g&\Ker{gN_\Ud^{i-1}}\ar@{}[r]|-{\oplus }\ar[d]^g&\wt{\Ud_{i}'}\ar[d]^{g}\\
  \Ker{fN_\Vd^{i-1}}\ar@{}[r]|{\simeq }& \Ker{N_\Vd^{i-1}}\ar@{}[r]|-{\oplus }&\VT{i-1},
}\qquad 
\xymatrix{    \Ker{gN_\Ud^i}\ar@{}[r]|{\simeq }\ar[d]^{g}& \Ker{N_\Ud^i}\ar@{}[r]|-{\oplus }\ar[d]^{g}&\UT{i}\ar[d]^{g}\\
\Ker{N_\Vd^i}\ar@{}[r]|{\simeq }\ar[d]^f&\Ker{fN_\Vd^{i-1}}\ar@{}[r]|-{\oplus }\ar[d]^f&\wt{\Vd_{i}'}\ar[d]^{f}\\
  \Ker{gN_\Ud^{i-1}}\ar@{}[r]|{\simeq }& \Ker{N_\Ud^{i-1}}\ar@{}[r]|-{\oplus }&\UT{i-1}.
}
\end{aligned}
\end{equation}

\end{lem}

\begin{proof}
Fix decompositions (\ref{mar4-6}).
We take $i_0\geq 1$ such that $N_\Vd^{i_0}=0$ and $N_\Ud^{i_0}=0$.
In particular, we have 
${\VT{i_0}}=0$,
${\VTp{i_0+1}}=0$,
${\UT{i_0}}=0$ and
${\UTp{i_0+1}}=0$.
We choose a section $s\colon {\UTp{i_0}}\to \Ker{N_\Ud^{i_0}}$ of $\Ker{N_\Ud^{i_0}}\to {\UTp{i_0}}$.
Moreover, we choose an arbitrary section $t\colon \VT{\Iz-1}\to \Ker{fN_\Vd^{\Iz-1}}$ of $\Ker{fN_\Vd^{\Iz-1}}\to \VT{\Iz-1}$.
Then,
using the decomposition 
$\VT{\Iz-1}\simeq g(\UTp{\Iz})\oplus \Vd_{\Iz-1}$,
we define a modified section $\wt{t}\colon \VT{\Iz-1}\to \Ker{fN_\Vd^{\Iz-1}}$ as
\[\VT{\Iz-1}\simeq g(\UTp{\Iz})\oplus \Vd_{\Iz-1}
\simeq  \UTp{\Iz}\oplus \Vd_{\Iz-1}
\xrightarrow{(gs,t)}
\Ker{fN_\Vd^{\Iz-1}}.\]
By the construction, we obtain a commutative diagram:
\[
\xymatrix{\Ker{N_\Ud^{\Iz}}\ar[d]^g& \UTp{\Iz}\ar[d]^g\ar[l]_-{s}\\
\Ker{fN_\Vd^{\Iz-1}}& \VT{\Iz-1}\ar[l]_-{\wt{t}}}
\]
By repeating this procedure, we obtain the desired sections.
\end{proof}

\begin{rem}
By the proof of Lemma~\ref{mar4-8},
after choosing decompositions~(\ref{mar4-6}) for all $i$
and 
arbitrary sections of the quotient maps appearing in Lemma~\ref{mar4-8},
we can modify these sections inductively so that the diagrams obtained by adding these sections commute.
This construction will be used in Section~\ref{sectionmhm}.
\end{rem}

\begin{cor}\label{ap3-3-c}%corからlemにyapparicorに
    We fix decompositions as in (\ref{mar4-6}) and (\ref{ap1-2}).
We regard
$\VT{i}$ and $\Vd_i$ (resp. $\UTp{i}$ and $\Ud_i'$) as subspaces of $\Vd$ (resp. $\Ud$).
Then, we obtain the following commutative diagram:
\begin{equation}
    \label{ap3-2}\begin{aligned}
    \xymatrix{
\Vd\ar[r]^f& \Ud\ar[r]^{g}&\Vd\ar[r]^f&\cdots \ar[r]^f& \Ud\ar[r]^g& \Vd \ar[r]^f& \Ud\\
\VT{i}\ar@{^{(}->}[u]\ar[r]^f&\UTp{i}\ar@{^{(}->}[u]\ar[r]^g&
\VT{i-1}\ar@{^{(}->}[u]\ar[r]^f&\cdots \ar[r]^f&\UTp{1}\ar@{^{(}->}[u]\ar[r]^g&\VT{0}\ar@{^{(}->}[u]\ar[r]^f&0\\
\Vd_{i}\ar@{^{(}->}[u]\ar[r]^f&f(\Vd_i)\ar@{^{(}->}[u]\ar[r]^g&
N_\Vd(\Vd_i)\ar@{^{(}->}[u]\ar[r]^f&\cdots \ar[r]^f&fN_\Vd^{i-1}(\Vd_i)\ar[r]^g\ar@{^{(}->}[u]&N_\Vd^{i}(\Vd_i)\ar@{^{(}->}[u]\ar[r]^f&0
}
\end{aligned}
\end{equation}
where 
\[
f(\Vd_i)\hookrightarrow \UTp{i},\quad
N_\Vd(\Vd_i)\hookrightarrow \VT{i-1},\quad \ldots
\]
are the inclusion maps induced by the inclusions
\[
f(\Vd_i)\subset f(\VT{i}),\quad
N_\Vd(\Vd_i)\subset g(\UTp{i}),\quad  \dots
\]
under the chosen decompositions in (\ref{mar4-6}).
An analogous commutative diagram is obtained from $\VTp{i}$, $\UT{i}$, $\Vd_i'$ and $\Ud_i$.
\end{cor}

\begin{proof}
    This immediately follows from Lemma~\ref{mar4-8}.
\end{proof}

By this corollary, 
the image $f(\Vd_i)$ of $\Vd_i$ under the induced morphism
$f\colon \VT{i}\to \UTp{i}$ is identified,
via the chosen decompositions (\ref{mar4-6}) and (\ref{ap1-2}),
with the subspace of $\Ud$ obtained by regarding $\Vd_i$ as a subspace of $\Vd$
and applying
$f\colon \Vd\to \Ud$.
Hence, we shall not distinguish between these two subspaces in what follows.
The same convention applies to $N_\Vd(\Vd_i)$, $g(\Ud_i)$, and similar subspaces.

\begin{lem}\label{ap8-3}%\say{これはdirect sumなので無限和にしないでどこかでとめる。}
    We fix decompositions as in (\ref{mar4-6}) and (\ref{ap1-2}).
As in Corollary~\ref{ap3-3-c},
we regard $\bigoplus_{r=0}^iN_\Vd^r(\Vd_i)$
and $\bigoplus_{r=0}^{i-1}fN_\Vd^r(\Vd_i)$
as subspaces of $\Vd$ and $\Ud$,
respectively,
with the convention that $\bigoplus_{r=0}^{-1}=0$.
Then, we obtain a subdiagram of (\ref{jan1-1})
\begin{align*}
    \xymatrix{    \displaystyle\bigoplus_{r= 0}^i N_\Vd^r(\Vd_i)\ar@<0.5ex>[r]^-{f}&
    \displaystyle\bigoplus_{r= 0}^{i-1}fN_\Vd^r (\Vd_i)
    \ar@<0.5ex>[l]^-{g}}.
\end{align*}
Moreover, this is a direct summand of (\ref{jan1-1}).
The same assertions hold for $\Vd_i'$, $\Ud_i$ and $\Ud_i'$.
\end{lem}

\begin{proof}
%By the previous lemma,
%the morphism
%$
%N_{\Vd}^\ell(\Vd_i)\to fN_{\Vd}^\ell(\Vd_i)$
%(resp. $fN_{\Vd}^\ell(\Vd_i)\to N_{\Vd}^\ell(V_{i-1})$)
%in the diagram is the restriction of
%$f\colon V\to W$ (resp. $g\colon W\to V$),
%where both sides are regarded as subspaces of $V$ and $W$.

We may assume $i\geq 1$.
Applying $f$ to the decomposition
$\VT{i}\simeq g(\UTp{i+1})\oplus \Vd_i$,
we have
\begin{align}\label{ap7-1}
    f(\VT{i})\simeq N_\Ud(\UTp{i+1})\oplus f(\Vd_i)
\end{align}
and hence
\[\UTp{i}\overset{(\ref{mar4-6})}{\simeq} \Ud'_i\oplus f(\VT{i})
\simeq \Ud'_i\oplus N_\Ud(\UTp{i+1})\oplus f(\Vd_i).
\]
Together with (\ref{ap1-2}) and Corollary~\ref{ap3-3-c}, 
\[ \xymatrix{\Vd_i\ar[r]^-f&f(\Vd_i)}\]
is a direct summand of  
\[ \xymatrix{\Vd\ar[r]^-f&\Ud}\]
as a diagram.

Similarly, applying $g$ to (\ref{ap7-1})
and $\UTp{i}\simeq f(\VT{i})\oplus \Ud'_i$,
we have
\begin{align*}
    N_\Vd(\VT{i})&\simeq gN_\Ud(\UTp{i+1})\oplus N_\Vd(\Vd_i),\mbox{and}\\
        g(\UTp{i})&\simeq N_\Vd(\VT{i})\oplus g(\Ud'_i).
\end{align*}
Therefore, 
we have
\[
\VT{i-1}\overset{(\ref{mar4-6})}{\simeq} \Vd_{i-1}\oplus g(\UTp{i})
\simeq \Vd_{i-1}\oplus gN_\Ud(\UTp{i+1})\oplus N_\Vd(\Vd_i) \oplus g(\Ud'_i).
\]
Together with (\ref{ap1-2}) and Corollary~\ref{ap3-3-c}, 
\[ \xymatrix{N_\Vd(\Vd_i)& \ar[l]_-g f(\Vd_i)}\]
is a direct summand of 
\[ \xymatrix{\Vd& \ar[l]_-g \Ud}\]
as a diagram.

Repeating this procedure, we obtain the desired assertion.
\end{proof}

The following corollary follows immediately from Lemma~\ref{ap8-3}.
\begin{cor}\label{ap8-4}
We have the following direct sum decomposition of the diagram (\ref{jan1-1}):
\begin{equation}
    \label{ap8-6}
\begin{aligned}
 &\bigoplus_{i\geq 0}\left(   \xymatrix{ \displaystyle\bigoplus_{r=0}^{i}N_\Vd^r(\Vd_i)\ar@<0.5ex>[r]^-{f}&
    \displaystyle\bigoplus_{r=0}^{i-1}fN_\Vd^r (\Vd_i)
    \ar@<0.5ex>[l]^-{g}
    }\right)\\
\oplus  &\displaystyle\bigoplus_{i\geq 1}\left(   \xymatrix{ \displaystyle\bigoplus_{r= 0}^{i-1}N_\Vd^r(\Vd_i')\ar@<0.5ex>[r]^-{f}&
    \displaystyle\bigoplus_{r= 0}^{i-1}fN_\Vd^r (\Vd_i')
    \ar@<0.5ex>[l]^-{g}
    }\right)\\
    \oplus 
&\bigoplus_{i\geq 0}   
\left(    \xymatrix{
 \displaystyle\bigoplus_{r=0}^{i-1}gN_\Ud^r(\Ud_i)\ar@<0.5ex>[r]^-{f}&
    \displaystyle\bigoplus_{r=0}^iN_\Ud^r (\Ud_i)
    \ar@<0.5ex>[l]^-{g}}\right)\\
    \oplus 
    &\bigoplus_{i\geq 1}   
\left(    \xymatrix{\displaystyle
 \bigoplus_{r= 0}^{i-1}gN_\Ud^r(\Ud_i')\ar@<0.5ex>[r]^-{f}&
    \displaystyle\bigoplus_{r= 0}^{i-1}N_\Ud^r (\Ud_i')
    \ar@<0.5ex>[l]^-{g}}\right).
\end{aligned}
\end{equation}

Here, only finitely many summands indexed by $i$ are non-zero.
\end{cor}

For $x\in \Vd_i$, we denote by $\La x\Ra$ the subspace of $\Vd$ generated by $x$.
Then, by the previous corollary, 
the diagram
\[\xymatrix{
 \bigoplus_{r=0}^iN_\Vd^r\La x\Ra\ar@<0.5ex>[r]^-{f}&
    \bigoplus_{r=0}^{i-1}fN_\Vd^r \La x\Ra
    \ar@<0.5ex>[l]^-{g}}
\]
is a direct summand of (\ref{jan1-1}) as a diagram.
We denote it by $\diag(x)$.
We use the same notation for vectors in $\Vd_i'$, $\Ud_i$ and $\Ud_i'$.

\begin{cor}\label{ap7-2}
The diagram (\ref{jan1-1}) is isomorphic to a direct sum of diagrams of the form $\diag(x)$,
where $x$ ranges over the elements of chosen bases of the spaces
$\Vd_i$, $\Vd_i'$, $\Ud_i$, and $\Ud_i'$ for all $i$.
\end{cor}

\begin{defi}\label{sep1-1}
\begin{enumerate}
    \item For each $i\geq 1$,
we say that a diagram is of type $A_i$ (resp. $B_i$, $C_i$, $D_i$) if it is isomorphic to $\diag(x)$ for some non-zero vector $x$ in $\Ud_i'$ (resp. $\Vd_i'$, $\Ud_{i-1}$, $\Vd_i$).
We define type $D_0$ similarly, using a non-zero vector $x\in \Vd_0$.
See the diagrams of types $A_i$, $B_i$, $C_i$, and $D_i$ below.

\item By abuse of notation, we denote by $A_i$ (resp. $B_i$, $C_i$, $D_i$) the (unipotent and monodromic) perverse sheaf on $\CC_t$ corresponding to a diagram of type $A_i$ (resp. $B_i$, $C_i$, $D_i$).
\end{enumerate}

\end{defi}

The indexing is chosen so that assertion \ref{mar3-4} in Proposition~\ref{mar3-2} holds.
The diagrams of these types are as follows.

\begin{center}
\centering

\begin{minipage}{0.24\textwidth}
\centering
\parbox[c][5.0cm][c]{\linewidth}{%
\[
\xymatrix@R=1pt{
&\bullet_{x}\ar[ld]\\
\bullet_{gx}\ar[dr]&\\
&\bullet_{Nx}\ar[ld]\\
\bullet_{gNx}\ar[dr]&\\
&\bullet_{N^2x}\\
\ar@{}[r]|-{\dots}&\\
&\bullet_{N^{i-1}x}\ar[dl]\\
\bullet_{gN^{i-1}x}\ar@{.>}[rd]&\\
&0\\
}
\]
}
\par
{\small Type $A_i$}
\end{minipage}
\hfill
\begin{minipage}{0.24\textwidth}
\centering
\parbox[c][5.0cm][c]{\linewidth}{%
\[
\xymatrix@R=1pt{
\bullet_{x}\ar[dr]&\\
&\bullet_{fx}\ar[ld]\\
\bullet_{Nx}\ar[dr]&\\
&\bullet_{fNx}\ar[ld]\\
\bullet_{N^2x}\ar@{}[rd]&\\
\ar@{}[r]|-{\dots}&\\
\bullet_{N^{i-1}x}\ar[rd]&\\
&\bullet_{fN^{i-1}x}\ar@{.>}[ld]\\
0&\\
}
\]
}
\par
{\small Type $B_i$}
\end{minipage}
\hfill
\begin{minipage}{0.24\textwidth}
\centering
\parbox[c][5.0cm][c]{\linewidth}{%
\[
\xymatrix@R=1pt{
&\bullet_x\ar[ld]\\
\bullet_{gx}\ar[dr]&\\
&\bullet_{Nx}\ar[ld]\\
\bullet_{gNx}\ar[dr]&\\
&\bullet_{N^2x}\ar@{}[ld]|-{\dots}\\
\bullet_{gN^{i-2}x}\ar[rd]&\\
&\bullet_{N^{i-1}x}\ar@{.>}[ld]\\
0&\\
}
\]
}
\par
{\small Type $C_i$}
\end{minipage}
\hfill
\begin{minipage}{0.24\textwidth}
\centering
\parbox[c][5.0cm][c]{\linewidth}{%
\[
\xymatrix@R=1pt{
\bullet_{x}\ar[dr]&\\
&\bullet_{fx}\ar[ld]\\
\bullet_{Nx}\ar[dr]&\\
&\bullet_{fNx}\ar[ld]\\
\bullet_{N^2x}\ar@{}[rd]|-{\dots}&\\
&\bullet_{fN^{i-1}x}\ar[ld]\\
\bullet_{N^ix}\ar@{.>}[rd]&\\
&0
}
\]
}
\par
{\small Type $D_i$}
\end{minipage}

%\caption{}
%\label{fig:four}
\end{center}

We write $\taik_\CC$ for the constant sheaf, and $\taik_0$ for the skyscraper sheaf at $0$.
We write $\taik_{\CS}$ for the constant sheaf on $\CS$ and, by the same symbol, for its zero extension to $\CC$.
Let $\calL_i$ be a local system on $\CS$ whose monodromy has a single Jordan block of size $i$ with eigenvalue $1$.
We denote by $j\colon \CS\hookrightarrow \CC$ the inclusion.
Then, from the definitions, we have the following.

\begin{prop}\label{mar3-2}
\begin{enumerate}
    \item $A_1\simeq \taik_{\CS}[1]$,\ $B_1\simeq \Derived j_*\taik_\CS[1](=\Derived\Gamma_{\CS}\taik_\CC[1])$, $C_1\simeq \taik_{\CC}[1]$, $D_0\simeq \taik_0$,
    \item $A_i\simeq j_!\calL_i[1]$, $B_i\simeq \Derived j_*\calL_i[1]\;(=\Derived\Gamma_{\CS}(j_{!}\calL_i[1]))$,
    \item $C_i\simeq j_{!*}\calL_i[1]$, where $j_{!*}\calL_i[1]$ denotes the minimal extension of $\calL_i[1]$,
    \item $D_i\simeq \cokaa(\taik_\CC[1]\to \Derived j_*\calL_{i+1}[1])\simeq \kaa(j_!\calL_{i+1}[1]\to \taik_\CC[1])$.
\item \label{mar3-4}Moreover, for $X\in \{A,B,C,D\}$,
we have
\[X_i|_{\CS}\simeq \calL_i[1],\]
except for $D_0(=\taik_0)$, in which case the restriction is zero.
\end{enumerate}

\end{prop}

By Corollary~\ref{ap7-2},
we obtain the following.

\begin{prop}\label{aug5-10}
    For any $P\in \mathrm{Perv}_{\mon,\unip}(\taik_{\CC_t})$,
    $P$ can be decomposed as
    \begin{align*}
        P\simeq \bigoplus_{s=1}^{\ell}X_{i(s)}^s,
    \end{align*}
    where $X^s\in \{A,B,C,D\}$ and $i(s)\in \ZZ_{\geq 0}$,
    with $i(s)\geq 1$ unless $X^s=D$.
    \end{prop}

\subsection{The case without eigenvalue $1$}\label{aug5-2}

Let $\Perv_{\mon,\neq 1}(\taik_{\CC_t})$
be the full subcategory of $\Perv_{\mon}(\taik_{\CC_t})$
consisting of 
objects $P\in \Perv_{\mon}(\taik_{\CC_t})$ with
$\psi_{t,1}P=\phi_{t,1}P=0$.
Such an object $P$ is the zero extension $j_{!}(P|_{\CS})$ of the restriction $P|_{\CS}$.
Thus, the category $\Perv_{\mon,\neq 1}(\taik_{\CC_t})$ is simpler than $\Perv_{\mon,\unip}(\taik_{\CC_t})$.
This category is equivalent to the category
of pairs $(U,T)$,
where $U$ is a finite-dimensional $\taik$-vector space
and $T$ is an automorphism of $U$ having no eigenvalue $1$.
Therefore, using the rational canonical form of $T$,
we obtain the following.

\begin{prop}\label{june14-1}
For any $P\in \mathrm{Perv}_{\mon,\neq 1}(\taik_{\CC_t})$,
there exist
an integer $\ell\geq 1$ and 
local systems $\calS_i$ on $\CS$ for $1\leq i\leq \ell$
such that the monodromy of each $\calS_i$ is semisimple and has no eigenvalue $1$, and
\[P\simeq \bigoplus_{i=1}^\ell
j_!(\calL_i\otimes \calS_i)[1].
\]
Here, $\calL_i$ denotes the local system introduced in the previous subsection.
\end{prop}

For later use in Section~\ref{sectionmhm},
we reinterpret Proposition~\ref{june14-1} from the viewpoint of the classification in the previous subsection.
Let $T=T_sT_u$ be the multiplicative Jordan--Chevalley decomposition of $T$,
where $T_s$ and $T_u$ are its semisimple and unipotent parts, respectively.
Setting $\wt{N}:=\log{T_u}$,
giving the pair $(U,T)$ is equivalent to 
giving a diagram
\begin{align}\label{jun13-2}
    \xymatrix{    (U,T_s)\ar@<0.5ex>[r]^-{\wt{N}}&
(U,T_s)   \ar@<0.5ex>[l]^-{=}}.
\end{align}
After forgetting the $T_s$-action, this is a diagram of the form considered in the previous subsection, with both vertices equal to $U$, $f=\wt{N}$, and $g=\iden$.
In this case, the subquotients $\Vd_i$, $\Vd_i'$, and $\Ud_i$ are all zero.
Only $\Ud_i'$ can be non-zero.

\begin{lem}\label{jun13-8}
Let $\calA$ be an abelian category,
and let
\begin{align}\label{jun13-1}
    0\to A'\to A\to A''\to 0
\end{align}
be an exact sequence in $\calA$.
Assume that $A$, $A'$, and $A''$ are equipped with automorphisms $T$, $T'$, and $T''$, respectively,
such that all the morphisms in (\ref{jun13-1}) are equivariant.
Assume that (\ref{jun13-1}) is split in $\calA$ and
that one of the following conditions holds:
\begin{enumerate}
\item $\calA$ is the category of finite-dimensional vector spaces over a field $\taik$, and $T$ is semisimple.
\item $\calA$ is a $\QQ$-linear category, and $T^m=\iden$ for some $m\geq 1$.
\end{enumerate}
Then,
there exists an isomorphism
$A\simeq A'\oplus A''$ in $\calA$
under which $T$ corresponds to $T'\oplus T''$.
\end{lem}

\begin{proof}
    If the first condition holds,
    the assertion follows from the
    rational canonical form of the semisimple automorphism $T$.
  
    Suppose that the second condition holds.
Take a section $s\colon A''\to A$ of the morphism $A\to A''$ in (\ref{jun13-1}).
Then, the section 
\[s'=\dfrac{1}{m}\sum_{i=0}^{m-1}T^{-i}s(T'')^{i}\]
of $A\to A''$ is equivariant with respect to $T$ and $T''$ and gives the desired equivariant splitting.    
\end{proof}

Since $T_s$ commutes with $\wt{N}$, 
all the subquotients appearing in the previous subsection are stable under $T_s$.
Moreover, by this lemma,
all the splittings from Lemma~\ref{jan7-2} through Corollary~\ref{ap8-4} can be chosen $T_s$-equivariantly.

Consequently,
we obtain the following direct sum decomposition of the diagram (\ref{jun13-2}) of vector spaces with semisimple automorphisms:
\begin{align*}
    \bigoplus_{i\geq 1}   
\left(    \xymatrix{\displaystyle
 \bigoplus_{r= 0}^{i-1}(\wt{N}^r\Ud_i',T_s)\ar@<0.5ex>[r]^-{\wt{N}}&
    \displaystyle\bigoplus_{r= 0}^{i-1}(\wt{N}^r\Ud_i',T_s)
    \ar@<0.5ex>[l]^-{=}}\right).
    \end{align*}
This decomposition recovers Proposition~\ref{june14-1}.
In Section~\ref{sectionmhm},
we will use the same argument to obtain its Hodge-theoretic analogue.

\input{section4}

\bibliographystyle{alpha}
\bibliography{reference}

\noindent Takahiro Saito:
Department of Mathematics, Chuo University,
1-13-27 Kasuga, Bunkyo-ku, Tokyo 112-8551, Japan.

\noindent
\textit{E-mail address}: \texttt{takahiro.saito27.a.gmail.com}

\end{document}

%% file: section4.tex
\section{Classification of monodromic mixed Hodge modules on $\CC$}
\label{sectionmhm}
\subsection{Mixed Hodge structures with fixed graded pieces}
\label{ssfixedgradedpieces}

Let us recall a standard fact about extensions of mixed Hodge structures.

\begin{prop}[For example, Theorem~3.31 of \cite{PetersSteen}]\label{aug19-1}
    Let $H=(H_\CC,F^\bullet H_\CC,H_\QQ,W_\bullet H_\QQ)$ and $H'=(H'_\CC,F^\bullet H'_\CC,H'_\QQ,W_\bullet H'_\QQ)$ be ($\QQ$-)mixed Hodge structures with $W_mH_\QQ=H_\QQ$ and $W_mH'_\QQ=0$.
\begin{enumerate}
    \item 
    The isomorphism classes of extensions of $H'$
    by $H$ of the form:
\begin{align}
    0\to {H}\to \wt{H}\to {H}'\to 0
\end{align}
are classified by
\begin{align}\label{ap15-1}
    \exten{{H}'}{{H}}:=\Hom_\CC(H',H)/\left(\Hom_{F}({H}',{H})+\Hom_{\QQ}(H',H)\right),
\end{align}
where 
$\Hom_{\CC}(H',H):=\Hom_{\CC}(H'_\CC,H_\CC)$
and
$\Hom_{F}(H',H)$ (resp. $\Hom_{\QQ}(H',H)$) denote the subgroups of $\CC$-linear maps $H'_\CC\to H_\CC$ preserving the Hodge filtration (resp. the $\QQ$-structure).

\item The mixed Hodge structure $\wt{H}$ 
which corresponds to $\Phi\in \Hom_\CC(H',H)$ is given by the following data:
\begin{itemize}
    \item the $\CC$-vector space $H_\CC\oplus H'_\CC$ and the $\QQ$-vector space $H_\QQ\oplus H'_\QQ$,
    \item the weight filtration $W_\bullet H_\QQ\oplus W_\bullet H'_\QQ$,
    \item the Hodge filtration $F^\bullet \wt{H}_\CC$ defined as 
    \[F^{p}\wt{H}_\CC=\{(x+\Phi(y),y)\in H_\CC\oplus H'_\CC\ |\ x\in F^pH_\CC, y\in F^pH'_\CC\}.\]
\end{itemize}
\end{enumerate}
\end{prop}

We now extend the preceding description to mixed Hodge structures with fixed graded pieces.
Let $\Lambda\subset \ZZ$ be a finite subset, equipped with the usual order.
Let
$\{L_w\}_{w\in \Lambda}$ be a family of pure Hodge structures,
where $L_w$ is of weight $w$.
We denote the underlying $\CC$-vector space (resp. $\QQ$-vector space) by $L_{w,\CC}$ (resp. $L_{w,\QQ}$).
We denote by $F^\bullet L_{w,\CC}$ the Hodge filtration of $L_w$. 

\begin{defi}\label{may14-3}
 \begin{enumerate}
     \item     
We denote by 
    $\Exten{\{L_w\}_{w\in \Lambda}}$ 
    the set of isomorphism classes of 
    pairs $(H,\{\alpha_w\}_{w\in \Lambda})$,
    where $H$ is a mixed Hodge structure
    and
    $\GR^W_wH=0$ for
$w\notin\Lambda$, and, for each $w\in \Lambda$,
\[\alpha_w\colon L_w\simto \GR^{W}_{w}H\]
is an isomorphism of pure Hodge structures.
Here an isomorphism
$(H,\{\alpha_w\})\simeq (H',\{\alpha'_w\})$ means an isomorphism
$f\colon H\simto H'$ of mixed Hodge structures
such that
    \[\GR^{W}_{w}f\circ \alpha_w =\alpha'_w,\]
for each $w\in \Lambda$.
\item 
We define $\UpT:=\UpT(\{L_w\}_{w\in \Lambda})$ to be the group of matrices
 $(m_{ij})_{i,j\in \Lambda}$
such that
\[m_{ij}\in \Hom_{\CC}(L_{j,\CC},L_{i,\CC})\text{,\quad and}\]
\[
    m_{ij}=\left\{\begin{array}{cc}
         \Id_{L_{i,\CC}}&  (i=j)\\
         0&  (i>j).
    \end{array}\right.
\]
In other words, the matrices in $\UpT(\{L_w\}_{w\in \Lambda})$ are upper
triangular with respect to the usual order of $\Lambda\subset \ZZ$.

\item We define subgroups $\UpT_F$ and $\UpT_\QQ$ of $\UpT$ by
\begin{align*}
    \UpT_F&:=
\left\{
(m_{ij})_{i,j\in \Lambda}\in \UpT \ \middle|\ 
m_{ij}(F^pL_{j,\CC})\subset F^pL_{i,\CC}
\quad (i,j\in \Lambda,\  p\in \ZZ)
\right\},\quad \text{and}\\
    \UpT_\QQ&:=
\left\{
(m_{ij})_{i,j\in \Lambda}\in \UpT \ \middle|\ 
m_{ij}(L_{j,\QQ})\subset L_{i,\QQ}
\quad (i,j\in \Lambda)
\right\}.
\end{align*}

\item For $M=(m_{ij})_{i,j\in \Lambda}\in \UpT$, 
we also denote by $M$ the $\CC$-linear automorphism of
$\bigoplus_{w\in \Lambda}L_{w,\CC}$
represented by the matrix $M$,
where elements of $\bigoplus_{w\in \Lambda}L_{w,\CC}$ are regarded as column vectors.
Explicitly,
for a vector $x=(x_i)_{i\in \Lambda}\in \bigoplus_{w\in \Lambda} L_{w,\CC}$,
it is given by
\[Mx=\left(\sum_{j\in \Lambda} m_{ij}(x_j)\right)_{i\in \Lambda}.\]

\item For $M\in \UpT$, we define a mixed Hodge structure $H_M$ as follows.
The underlying $\CC$-vector space $H_{M,\CC}$, $\QQ$-vector space $H_{M,\QQ}$ and the weight filtration $W_\bullet H_{M,\QQ}$ are the same as those of the mixed Hodge structure $\bigoplus_{w\in \Lambda}L_w$ obtained as the direct sum of the pure Hodge structures.
Its Hodge filtration is defined by
\[F^pH_{M,\CC}=M\left(\bigoplus_{w\in \Lambda}F^pL_{w,\CC}\right).\]

\end{enumerate}
\end{defi}

By construction, the associated graded object $\GR^W_{\bullet }H_{M}$ is naturally identified with $\bigoplus_{w\in \Lambda}L_{w}$.
Thus $H_M$, together with these identifications, defines an element of $\Exten{\{L_w\}_{w\in \Lambda}}$.

Let us consider the double quotient
\[\UpT_\QQ \backslash \UpT / \UpT_F,\]
where $\UpT_\QQ$ acts on $\UpT$ by left multiplication and $\UpT_F$ acts on $\UpT$
by right multiplication.
In other words, two elements $M,M'\in \UpT$ are equivalent if there exist
$A\in \UpT_\QQ$ and $B\in \UpT_F$ such that
\[M'=AMB.\]
Note that $\UpT_\QQ \backslash \UpT / \UpT_F$
has no natural group structure in general.
When $\Lambda$ consists of two elements $w<w'$,
the set $\UpT_\QQ \backslash \UpT / \UpT_F$ is identified with $\Hom_\CC(L_{w'},L_{w})/\left(\Hom_{F}(L_{w'},L_{w})+\Hom_{\QQ}(L_{w'},L_{w})\right)$,
i.e. $\exten{L_{w'}}{L_{w}}$.

\begin{prop}\label{jun16-2}
The map 
\[\UpT\ni M\mapsto H_M\in \Exten{\{L_w\}_{w\in \Lambda}}\]
induces a bijection
    \[\UpT_\QQ \backslash \UpT / \UpT_F\simto \Exten{\{L_w\}_{w\in \Lambda}}. \]
\end{prop}

\begin{proof}
The assertion follows by induction on $|\Lambda|$, using Proposition~\ref{aug19-1}.
\end{proof}

Next, we consider morphisms between mixed Hodge structures of the form $H_M$.
Recall that for a pure Hodge structure $H$ of weight $w$,
$H^{p,q}_\CC=F^pH_\CC\cap \ov{F^qH_\CC}$ and $H_\CC=\bigoplus_{p+q=w}H^{p,q}_\CC$.

\begin{defi}\label{may18-4}
    \begin{enumerate}
        \item For pure Hodge structures $H$ and $H'$ of weights $w$ and $w'$, respectively, with $w>w'$,
    let $\varphi\colon H_\CC\to H'_\CC$ be a $\CC$-linear map.
    We say that $\varphi$ satisfies condition $(*)$ if
    \[\varphi({H^{p,q}})\subset \displaystyle\bigoplus_{\substack{p'<p\\ p'+q'=w'}}(H'_\CC)^{p',q'}=\overline{F^{w'-p+1}H'_{\CC}}\]
for any $p,q\in \ZZ$ with $p+q=w$.
    \item 
For $M=(m_{ij})\in \UpT(\{L_w\}_{w\in \Lambda})$, we say that $M$ satisfies condition $(*)$
if $m_{ij}\colon L_{j,\CC}\to L_{i,\CC}$ satisfies condition $(*)$
for any $i,j\in \Lambda$ with $i<j$.

    \end{enumerate}
    \end{defi}

\begin{lem}\label{may7-1}
    For a class in $\UpT_\QQ\backslash \UpT/\UpT_F$,
    we can always take a representative $M=(m_{ij})\in \UpT$
    satisfying condition $(*)$.
\end{lem}

%    \begin{enumerate}
 %       \item $m_{ij}({H^{p,q}L_{j,\CC}})\subset \displaystyle\bigoplus_{\substack{p'<p\\ p'+q'=i}}H^{p',q'}L_{i,\CC}=\overline{F^{i-p+1}L_{i,\CC}}$ for any $i,j,p,q\in \ZZ$ with $i<j$ and $p+q=j$,
  %      where $H^{p,q}$ is the $(p,q)$-part of a Hodge structure.
  %      \item $m_{ij}(F^pL_{j,\CC})\subset \overline{F^{i-p+1}L_{i,\CC}}$ for any $i,j,p\in \ZZ$ with $i<j$.
   % \end{enumerate}
\begin{proof}
    Take a representative $M\in \UpT$.
   For $i,j\in \Lambda$ with $i<j$ 
   we define a $\CC$-linear map $r_{ij}$ as follows.
   For $p,q\in \ZZ$ with $p+q=j$,
let $\pi_i^{\geq p}$ be the projection 
$L_{i,\CC}\to  \bigoplus_{\substack{p'\geq p\\ p'+q'=i}}L^{p',q'}_{i,\CC}(=F^pL_{i,\CC})$,
and set
\[r_{ij}|_{L^{p,q}_{j,\CC}}=\pi_i^{\geq p}\circ m_{ij}|_{L^{p,q}_{j,\CC}}.\]
   Then, $r_{ij}$ preserves the Hodge filtration, and $m_{ij}-r_{ij}$ satisfies condition $(*)$.
We prove the assertion by induction on an integer $s\geq 1$.
Assume that, after replacing $M$ by an equivalent representative,
$m_{ij}$ satisfies condition $(*)$ for all $i,j\in \Lambda$ with $0<j-i\leq s-1$.
This is true for $s=1$.
We show that after replacing $M$ by an equivalent representative,
condition $(*)$ also holds for $j-i=s$.
Let $B=(\beta_{ij})\in \UpT_F$ be the matrix defined by
$\beta_{ii}=\iden_{L_{i,\CC}}$ and
$\beta_{ij}=-r_{ij}$ if $j-i=s$,
and $\beta_{ij}=0$ for other $i<j$.
We consider $MB$.
For $a,b\in \Lambda$ with $b>a$,
the $(a,b)$-component of $MB$ is 
$\sum_{a\leq c\leq b}m_{ac}\beta_{cb}$.
If $b-a<s$, then this is equal to $m_{ab}$, which satisfies condition $(*)$ by the induction hypothesis.
If $b-a=s$, then the $(a,b)$-component is 
\[m_{ab}+m_{a(b-s)}\beta_{(b-s)b}=m_{ab}-r_{ab},\]
which satisfies condition $(*)$.
Thus the assertion holds for any $s$.
\end{proof}

Let $\{L_w'\}_{w\in \Lambda}$ be another family of pure Hodge structures,
where $L_w'$ is of weight $w$.
Set $\UpT':=\UpT(\{L_w'\}_{w\in \Lambda})$
and define $\UpT'_F$, $\UpT'_\QQ$ as in Definition~\ref{may14-3}.
For $i, i_0\in \ZZ$ with $i\leq i_0$, we set
\[
   \Lambda_{[i,i_0]}:=\{j\in \Lambda\mid i\leq j\leq i_0\}.
\]
We use similar notation such as $\Lambda_{<i_0}$, $\Lambda_{\leq i_0}$, and $\Lambda_{(i,i_0)}$.

\begin{prop}\label{sep28-1}
Let $M=(m_{ij})_{i,j\in \Lambda}\in \UpT$ and $M'=(m_{ij}')_{ i,j\in \Lambda}\in \UpT'$
be elements satisfying condition $(*)$.
%We denote by $[M]$ and $[M']$ their classes in
%$\UpT_{\QQ}\backslash \UpT/\UpT_F$ and
%$\UpT'_{\QQ}\backslash \UpT'/\UpT'_F$, respectively.
Let $f_w\colon L_{w,\CC}\to L'_{w,\CC}$ be the 
$\CC$-linear map underlying a morphism $L_w\to L'_w$ of pure Hodge structures.
Then, 
the $\CC$-linear map $f=\bigoplus_{w}f_w\colon \bigoplus_{w}L_{w,\CC}\to \bigoplus_{w}L'_{w,\CC}$
underlies a morphism of mixed Hodge structures $H_M\to H_{M'}$
if and only if we have
\[f_im_{ij}=m_{ij}'f_j\]
for $i,j\in \Lambda$ with $i\leq j$.
\end{prop}

\begin{proof}
Since each $f_w$ underlies a morphism of pure Hodge structures,
$f$ preserves the $\QQ$-structure and the weight filtration.
Thus, it suffices to check that $f$ preserves the Hodge filtration. 
By definition,
$f$ preserves the Hodge filtration of $H_M$ and $H_{M'}$ if and only if
for any $i_0\in \Lambda$ and $x_{i_0}\in F^pL_{i_0,\CC}$,
there exists $(y_j)\in \bigoplus_{j\in \Lambda}F^pL'_{j,\CC}$
such that
\[M'(y_j)=fM\iota_{i_0}(x_{i_0}),\]
where $\iota_{i}\colon L_{i,\CC}\hookrightarrow \bigoplus_{w\in \Lambda}L_{w,\CC}$
is the natural inclusion.
By the upper triangularity of $M$ and $M'$,
$y_j=0$ for $j>i_0$ and $y_{i_0}=f_{i_0}(x_{i_0})$.
Therefore, we have the equality
\begin{align}
    \label{may15-1}
    f_i m_{i i_0}(x_{i_0})=\sum_{\substack{j\in \Lambda_{[i,i_0]}}}m'_{ij}(y_j)
\end{align}
for $i\in \Lambda_{\leq  i_0}$.

Assume $f$ preserves the Hodge filtration of $H_M$ and $H_{M'}$
and $x_{i_{0}}\in L^{p,q}_{i_0,\CC}$ for $p+q=i_0$,
and take $(y_j)$ as above.
For fixed $i_1\in \Lambda_{<i_0}$,
suppose that
$y_j=0$ for any $j\in \Lambda_{(i_1,i_0)}$.
Then, by (\ref{may15-1}), we obtain
\[f_{i_1}m_{i_1i_0}(x_{i_0})=m'_{i_1i_0}f_{i_0}(x_{i_0})+y_{i_1}.\]
By condition $(*)$ for $M$ and $M'$,
$f_{i_1}m_{i_1i_0}(x_{i_0})$ and $m'_{i_1i_0}f_{i_0}(x_{i_0})$
are in $\overline{F^{i_1-p+1}L'_{i_1,\CC}}$.
Since $F^pL'_{i_1,\CC}\cap \overline{F^{i_1-p+1}L'_{i_1,\CC}}=\{0\}$ and $y_{i_1}\in F^pL'_{i_1,\CC}$,
we have $f_{i_1}m_{i_1i_0}(x_{i_0})-m'_{i_1i_0}f_{i_0}(x_{i_0})=0$ and $y_{i_1}=0$.
By descending induction on $i_1$,
for $i\in \Lambda_{<i_0}$ we obtain $y_{i}=0$ and
$f_im_{ii_0}(x_{i_0})=m'_{ii_0}f_{i_0}(x_{i_0})$.

Conversely,
assume that $f_im_{ij}=m_{ij}'f_j$ holds.
For any $x_{i_0}\in L^{p,q}_{i_0,\CC}$,
if we set $y_{i_0}=f_{i_0}(x_{i_0})(\in (L'_{i_0,\CC})^{p,q})$ and $y_j=0$ for $j\neq i_0$,
then we have $M'(y_j)=fM\iota_{i_0}(x_{i_0})$.
Since this holds for every $p,q$,
by linearity and the Hodge decomposition,
the same conclusion holds for any $x_{i_0}\in F^pL_{i_0,\CC}$.
This implies $f$ preserves the Hodge filtration.

This completes the proof.
\end{proof}

\subsection{Split graded models}

%We denote by $\MHSp$ the category of graded polarizable mixed Hodge structures.
%Since we will only consider graded polarizable mixed Hodge structures in this paper,
%we will simply call them mixed Hodge structures.

By Corollary~\ref{june11-3},
classifying the isomorphism classes of $\MHM_{\mathrm{mon},\unip}(\CC_t)$ is equivalent to classifying diagrams of mixed Hodge structures of the following form.
Let $\Vd$ and $\Ud$ be mixed Hodge structures,
and $f\colon \Vd\to \Ud(-1)$ and $g\colon \Ud\to \Vd$ be morphisms of mixed Hodge structures such that $gf$ and $fg$ are nilpotent:
\begin{align}\label{mar3-3}
    \xymatrix@C=40pt{\Vd\ar@<0.5ex>[r]^-{f,\ (-1)}&\Ud\ar@<0.5ex>[l]^-{g}}.
\end{align}
We identify the underlying $\QQ$-vector space of $\Ud(-1)$ with $\Ud_\QQ$ by multiplication by $2\pi\sqrt{-1}$.
In this way, the diagram above gives a diagram of $\QQ$-vector spaces of the form (\ref{jan1-1}), which we call the underlying diagram.

%Then, by ??, we obtain a monodromic mixed Hodge module on $\CC_t$
%whose underlying perverse sheaf corresponds to 
%the induced diagram on the underlying $\mathbb{Q}$-vector spaces.

Let $\QQ^H$ be the trivial polarizable pure Hodge structure of weight $0$,
whose Hodge filtration jumps at $0$.
For $i\geq 1$, we define a mixed Hodge structure $K_i$ by
\[K_i:=\QQ^H(i-1)\oplus \QQ^H(i-2)\oplus \cdots \oplus \QQ^H.\]
Set $K_{0}=0$.
We set $\Vd=\Ud=K_i$ and 
define $g=\iden$ and $f\colon \Vd\to \Ud(-1)$ so that
the restriction of $f$ to each summand $\QQ^H(j)$
is the identity for $j\leq i-2$, and zero for $j=i-1$.
In this case, 
the underlying diagram
is of type $A_{i}$ defined in Definition~\ref{sep1-1}.
Similarly, for the pairs 
$(\Vd,\Ud)=(K_i(-1),K_i)$ (resp. $(K_{i-1},K_i)$, $(K_{i+1}(-1),K_i)$),
we obtain a diagram whose underlying diagram is of type $B_i$ (resp. $C_{i}$, $D_{i}$).
To summarize, the pairs $(V,U)$ for the four types are as follows:
\[\begin{array}{c|cccc}
     & A_i&B_i&C_i&D_i \\\hline
     V&K_i&K_i(-1)&K_{i-1}&K_{i+1}(-1)\\
     U&K_i&K_i&K_i&K_i
\end{array}.\]
Here $i\geq 1$ for the types $A_i$, $B_i$, $C_i$, whereas $i\geq 0$ for $D_i$.
See the diagrams below for examples of the types $A_3$, $B_3$, $C_3$ and $D_3$.
\begin{center}
\centering

\begin{minipage}{0.24\textwidth}
\centering
\parbox[c][5.5cm][c]{\linewidth}{%
\[
\xymatrix@C=25pt@R=3pt{
&\QQ^H\ar[ld]\\
\QQ^H\ar[dr]|-{\scriptscriptstyle (-1)}&\\
&\QQ^H(1)\ar[ld]\\
\QQ^H(1)\ar[dr]|-{\scriptscriptstyle(-1)}&\\
&\QQ^H(2)\ar[ld]\\
\QQ^H(2)\ar@{.>}[dr]&\\
&0\\
}
\]
}
\par
{\small Type $A_3$}
\end{minipage}
\hfill
\begin{minipage}{0.24\textwidth}
\centering
\parbox[c][5.5cm][c]{\linewidth}{%
\[
\xymatrix@C=25pt@R=3pt{
\QQ^H(-1)\ar[rd]|-{\scriptscriptstyle (-1)}&\\
&\QQ^H\ar[ld]\\
\QQ^H\ar[dr]|-{\scriptscriptstyle (-1)}&\\
&\QQ^H(1)\ar[ld]\\
\QQ^H(1)\ar[dr]|-{\scriptscriptstyle (-1)}&\\
&\QQ^H(2)\ar@{.>}[ld]\\
0&\\
}
\]
}
\par
{\small Type $B_3$}
\end{minipage}
\hfill
\begin{minipage}{0.24\textwidth}
\centering
\parbox[c][5.5cm][c]{\linewidth}{%
\[
\xymatrix@C=25pt@R=1pt{
&\QQ^H\ar[ld]\\
\QQ^H\ar[dr]|-{\scriptscriptstyle (-1)}&\\
&\QQ^H(1)\ar[ld]\\
\QQ^H(1)\ar[dr]|-{\scriptscriptstyle (-1)}&\\
&\QQ^H(2)\ar@{.>}[ld]\\
0&
}
\]
}
\par
{\small Type $C_3$}
\end{minipage}
\hfill
\begin{minipage}{0.24\textwidth}
\centering
\parbox[c][5.5cm][c]{\linewidth}{%
\[
\xymatrix@C=25pt@R=1pt{
\QQ^H(-1)\ar[dr]|-{\scriptscriptstyle (-1)}&\\
&\QQ^H\ar[ld]\\
\QQ^H\ar[dr]|-{\scriptscriptstyle (-1)}&\\
&\QQ^H(1)\ar[ld]\\
\QQ^H(1)\ar[dr]|-{\scriptscriptstyle (-1)}&\\
&\QQ^H(2)\ar[dl]\\
\QQ^H(2)\ar@{.>}[dr]&\\
&0
}
\]
}
\par
{\small Type $D_3$}
\end{minipage}
\end{center}

\begin{defi}\label{ap8-0}
For each type $X_i$ with $X\in \{A,B,C,D\}$,
we define $X_i^H$ to be the unipotent monodromic mixed Hodge module on $\CC_t$
        associated to the diagram of mixed Hodge structures constructed above whose underlying diagram is of type $X_i$.
%We denote by $D(X_i^H)$ the corresponding diagram of mixed Hodge structures.
\end{defi}

For $i\geq 1$, let $\calL_i^H$ denote the admissible graded-polarizable variation of mixed Hodge structure on $\CS$ introduced in the Introduction, whose underlying local system is $\calL_i$.
By construction, 
for each $X\in \{A,B,C,D\}$, we have
\[\psi_{t}X^H_i\simeq K_i.\]
Moreover, for $i\geq 1$, 
\[X^H_i|_{\CS}\simeq \calL^H_i[1].\]
Note that $X_i^H$ is not the only mixed Hodge module
whose underlying perverse sheaf is of type $X_i$.
We choose this normalization so that the properties above hold.
We will describe the other possibilities later.

For a mixed Hodge structure $H$,
the associated graded object with respect to the weight filtration
\[\GR^W_\bullet H=\bigoplus_{r\in \ZZ}\GR^W_r H\]
is again a mixed Hodge structure. 
Then, for the diagram (\ref{mar3-3}),
by taking the associated graded object of $U$ and $V$ with respect to the weight filtrations,
we obtain another diagram of mixed Hodge structures:
\begin{align}
 \label{mar4-2}
    \xymatrix@C=40pt{\GR^W_\bullet \Vd\ar@<0.5ex>[r]^-{\GR^W_\bullet f,\ (-1)}&\GR^W_{\bullet}\Ud\ar@<0.5ex>[l]^-{\GR^W_\bullet g}}.
\end{align}

\begin{defi}
    For $\calM\in \MHM_{\mon,\unip}(\CC_t)$ corresponding to the diagram (\ref{mar3-3}),
    we denote by $G(\calM)$ the object in $\MHM_{\mon,\unip}(\CC_t)$
    corresponding to the diagram (\ref{mar4-2}).
\end{defi}

\begin{rem}\label{ap20-1}
    For $r\in \ZZ$,
    the subobject $W_r\calM\subset \calM$ corresponds to
    the following diagram:
    \begin{align*}
    \xymatrix@C=40pt{W_r\Vd\ar@<0.5ex>[r]^-{f,\ (-1)}&W_{r-1}\Ud\ar@<0.5ex>[l]^-{g}}.
\end{align*}
Therefore, in general,
$G(\calM)$ does not coincide with $\GR^{W}_\bullet \calM$. 
\end{rem}

Let $H$ be a pure Hodge structure of weight $w$.
It defines a pure Hodge module of weight $w+1$ on $\CC_t$
associated with the constant variation of Hodge structure with fiber $H$.
For a mixed Hodge module $\calM$ on $\CC_t$, we denote by $\calM\otimes H$
the tensor product of $\calM$ with this constant Hodge module.

\begin{thm}\label{may15-5}
    For any $\calM\in \MHM_{\mon,\unip}(\CC_t)$,
    the mixed Hodge module $G(\calM)$ can be decomposed as
    \begin{align}\label{ap10-2}
        G(\calM)\simeq \bigoplus_{s=1}^{\ell}(X^s)_{i(s)}^{H}\otimes H^s,
    \end{align}
    where $X^s\in \{A,B,C,D\}$ and $i(s)\in \ZZ_{\geq 0}$,
    with $i(s)\geq 1$ unless $X^s=D$, 
    $(X^s)_{i(s)}^{H}$ is the mixed Hodge module defined in
Definition~\ref{ap8-0},
    and
    $H^s$ is a pure Hodge structure. %\say{regarding as a constant Hodge strみたいに言いたいが、そうするとシフトがおかしい}
\end{thm}

\begin{proof}
Replacing $\calM$ by $G(\calM)$, we may assume that $\calM=G(\calM)$.
Let (\ref{mar3-3}) be the corresponding diagram.
Since the morphisms $f$, $g$, $N_\Vd$ and $N_\Ud$ are morphisms of mixed Hodge structures,
$\wt{\Vd_i}$, $\wt{\Vd_i'}$, $\wt{\Ud_i}$ and $\wt{\Ud_i'}$
in Definition~\ref{mar4-3}
can naturally carry mixed Hodge structures.
Therefore,
by Lemma~\ref{jan7-2},
the same holds for ${\Vd_i}$, ${\Vd_i'}$, ${\Ud_i}$ and ${\Ud_i'}$.
We use the same symbols for these mixed Hodge structures.
Since we are assuming that $\calM=G(\calM)$,
for $L=\Vd_i$ or $\Ud_i$ or their variants with primes and tildes, or $\Ker{fN^j_\Vd}$,
$\Ker{gN_\Ud^j}$, $\Ker{N^j_\Vd}$,
$\Ker{N_\Ud^j}$,
we have
\begin{align}\label{mar4-5}
    \GR^W_\bullet L\simeq L.
\end{align}
Recall that 
any short exact sequence of polarizable pure Hodge structures splits.
Combining this fact with the property (\ref{mar4-5}),
we can lift Lemma~\ref{jan7-2},
Lemma~\ref{mar4-8}, Corollary~\ref{ap3-3-c},
Lemma~\ref{ap8-3} and Corollary~\ref{ap8-4}
to the setting of diagrams of mixed Hodge structures.
Therefore, it suffices to prove the assertion
for each direct summand of (\ref{ap8-6}).
Thus, we consider the diagram:
\begin{align}\label{ap8-7}    
\xymatrix@C=40pt{\displaystyle
 \bigoplus_{r= 0}^{i-1}gN_\Ud^r(\Ud_i')\ar@<0.5ex>[r]^-{f,(-1)}&
    \displaystyle\bigoplus_{r= 0}^{i-1}N_\Ud^r (\Ud_i')
    \ar@<0.5ex>[l]^-{g}.
    }
\end{align}
The other direct summands of (\ref{ap8-6}) corresponding to $\Vd_i$, $\Vd_i'$ and $\Ud_i$
can be treated in the same way.
By (\ref{mar4-5}), we may assume that $\Ud_i'$ is pure.
Note that $N_\Ud^r(\Ud'_i)\simeq \Ud_i'(r)$
and $gN_\Ud^r(\Ud'_i)\simeq \Ud_i'(r)$.
The components
\begin{align*}
    gN_\Ud^r(\Ud'_i)&\xrightarrow{f} N_\Ud^{r+1}(\Ud'_i)(-1)\quad  (0\leq r\leq i-2)\quad \text{and}\\ 
gN_\Ud^r(\Ud_i')&\xleftarrow{g} N_\Ud^r(\Ud_i')\quad (0\leq r\leq i-1)
\end{align*}
of (\ref{ap8-7}) are identified with the identity morphisms
of $\Ud_i'(r)$.
Moreover, the restriction of $f$ to $gN_{\Ud}^{i-1}(\Ud_i')$ is zero.
Hence,
the diagram (\ref{ap8-7})
is identified with
the diagram
\begin{align}\label{ap8-8}    
\xymatrix@C=40pt{ 
K_i\otimes \Ud_i'\ar@<0.5ex>[r]^-{f\otimes \iden,(-1)}&
   K_i\otimes \Ud_i',
    \ar@<0.5ex>[l]^-{g\otimes \iden}
    }
\end{align}
where $f\colon K_i\to K_i(-1)$ and
$g\colon K_i\to K_i$ are those defined just before Definition~\ref{ap8-0}.
By the convention fixed just before the statement of this theorem,
we also denote by $\Ud_i'$ the constant Hodge module on $\CC_t$ associated to the Hodge structure $\Ud_i'$.
Then, in this case,
we have
\begin{align*}
    \psi_{t}(A_i^H\otimes \Ud_i')&\simeq \psi_t(A_i^H)\otimes \Ud_i'
    \simeq K_i\otimes \Ud_i',\quad \mbox{and}\\
 \phi_{t}(A_i^H\otimes \Ud_i')&\simeq \phi_t(A_i^H)\otimes \Ud_i'\simeq K_i\otimes \Ud_i'.
\end{align*}
Hence, the mixed Hodge module corresponding to the diagram (\ref{ap8-8})
is isomorphic to $A_i^H\otimes \Ud_i'$.
This completes the proof.
\end{proof}
%ここは4/8に議論がある。non-characteristic inverseを暗に使ってる。

\begin{defi}\label{may21-1}
A split graded model is a diagram $G$ of mixed Hodge structures equipped with a finite direct sum decomposition
\[G=\bigoplus_{\alpha}G_\alpha\]
such that each $G_\alpha$ is a diagram corresponding to a mixed Hodge module of the form $X_i^H\otimes L$, 
where $X\in\{A,B,C,D\}$ and $L$ is a pure Hodge structure.
We also call a mixed Hodge module whose corresponding diagram is a split graded model a split graded model.
\end{defi}

The $\Vd$-part and the $\Ud$-part of a split graded model are direct sums of Tate twists of pure Hodge structures, 
and the morphisms $f$ and $g$ are direct sums of identity maps and zero maps.
Theorem~\ref{may15-5} implies that the diagram corresponding to $G(\calM)$ can be equipped with the structure of a split graded model.

\begin{comment}

Consider a morphism between exact sequences:
\begin{align*}
    \xymatrix{
0\ar[r]&\Ker{N^i_V}\ar[r]\ar[d]^{f}&\Ker{fN^i_V}\ar[r]\ar[d]^{f}&\wt{\Vd_i}\ar[r]\ar[d]^{f}&0\\
    0\ar[r]&\Ker{N_\Ud^{i-1}}(-1)\ar[r]&\Ker{N^i_W}(-1)\ar[r]&\wt{\Ud_i'}(-1)\ar[r]&0.
    }
\end{align*}
Recall that 
any short exact sequences of polarizable pure Hodge structures are split.
Combining this fact with the property (\ref{mar4-5}),
we obtain (non-canonical) decompositions with a commutative diagram:
\begin{align}
    \xymatrix{
    \Ker{fN^i_V}\ar@{}[r]|-{\simeq}\ar[d]^{f}& \Ker{N^i_V}\ar@{}[r]|-{\oplus }\ar[d]^{f}&\wt{\Vd_i}\ar[d]^{f}\\
    \Ker{N^i_W}(-1)\ar@{}[r]|-{\simeq}& \Ker{N^i_W}(-1)\ar@{}[r]|-{\oplus }&\wt{\Ud_i'}(-1)
    }    
\end{align}

\begin{align*}
    \wt{\Vd_i}\simeq \Vd_i\oplus g(\wt{\Ud_i'}),\quad 
    \wt{\Vd_i'}\simeq \Vd_i'\oplus g(\wt{\Ud_i}),\quad
    \wt{\Ud_i}\simeq \Ud_i\oplus f(\wt{\Vd_i'}),\quad 
    \wt{\Ud_i'}\simeq \Ud_i'\oplus f(\wt{\Vd_i}).
\end{align*}
Similarly, we have a decomposition:
\[\Ker{N_{\Vd}}\simeq \wt{V_0}\oplus \wt{V_0'}.\]
Then,

for a polarizable pure Hodge structure $L$ and its sub Hodge structure $L'$,
there exists a sub Hodge structure $L''$ with
\[L\simeq L'\oplus L''.\]

    Take an integer $i_0\geq 1$
\end{comment}

\subsection{Classification without the Hodge filtration}

To classify $\MHM_{\mon,\unip}(\CC_t)$ in general, we temporarily forget the Hodge filtrations.

We first recall the following elementary fact.
\begin{lem}\label{may8-1-1}
In this lemma,
by a filtered vector space $(H,W_\bullet H)$,
we mean a finite-dimensional vector space $H$ equipped with a finite increasing filtration $W_\bullet H$, 
i.e. $W_rH=0$ for $r\ll 0$ and $W_rH=H$
for $r\gg 0$.
\begin{enumerate}
    \item \label{may11-2}  For every filtered vector space $(H,W_\bullet H)$,
    there exists an isomorphism of filtered vector spaces $\eta\colon H\simto \GR^W_\bullet H$ such that $\GR^W_\bullet (\eta)=\iden$ on $\GR^W_\bullet H$.
    Here $\GR^W_\bullet H$ is equipped with the filtration
    $W_r\GR^W_\bullet H=\bigoplus_{r'\leq r}\GR^W_{r'} H$.
\item \label{may11-1}
Suppose that
\[0\to (H',W_\bullet H')\to (H,W_\bullet H)\to (H'',W_\bullet H'')\to 0\]
is an exact sequence of filtered vector spaces,
i.e. 
for every $r\in \ZZ$,
the sequence
\[0\to W_rH'\to W_r H\to W_rH''\to 0\]
is exact.
Then, the surjection $H\twoheadrightarrow H''$ admits a filtered section.
Consequently, there is a non-canonical decomposition
\[(H,W_\bullet H)\simeq (H',W_\bullet H')\oplus (H'',W_\bullet H'')\]
as filtered vector spaces.
\end{enumerate}
\end{lem}

The isomorphism $\eta$ in \ref{may11-2} of Lemma~\ref{may8-1-1} is not unique.
For a diagram (\ref{mar3-3}), we will choose such isomorphisms in a way compatible with the diagram.
Consider the underlying diagram
    \begin{align}\label{may8-4}
    \xymatrix@C=40pt{\Vd_\QQ\ar@<0.5ex>[r]^-{f,\ (-1)}&\Ud_\QQ\ar@<0.5ex>[l]^-{g}}.
\end{align}
In what follows, we only use the underlying $\QQ$-vector spaces.
By abuse of notation, we write $\Vd$ and $\Ud$ for $\Vd_\QQ$ and $\Ud_\QQ$,
respectively.
As observed in the proof of Theorem~\ref{may15-5},
the spaces $\Vd_i$, $\Ud_i$, their variants with primes and tildes, and the images under
$f$ or $g$ carry mixed Hodge structures.
In particular, each of these spaces is equipped with a weight filtration.
In Lemma \ref{may8-6} and Corollary \ref{may12-1} below, 
we work only with the underlying $\QQ$-vector spaces equipped with the weight filtrations, and
ignore the Hodge filtrations.
We note that $f$ induces a morphism $\wt{\Vd'_{i+1}}\to \wt{\Ud_i}(-1)$.
Its image $f(\wt{\Vd'_{i+1}})$ is therefore a subspace of $\wt{\Ud_i}(-1)$.
After applying the Tate twist $(1)$,
we regard $f(\wt{\Vd'_{i+1}})(1)$ as a subspace of $\wt{\Ud_i}$.
We adopt the same convention for the analogous spaces.

\begin{lem}\label{may8-6}
We can choose the decompositions (\ref{mar4-6}) and (\ref{ap1-2})
so that they are compatible with the weight filtrations.
More precisely,
we can choose the decompositions
$\wt{\Vd_{i}}\simeq g(\wt{\Ud_{i+1}'})\oplus \Vd_i$,
$\wt{\Ud_{i}}\simeq f(\wt{\Vd_{i+1}'})(1)\oplus \Ud_i$
so that they are isomorphisms of filtered vector spaces:
\begin{align*}
    (\wt{\Vd_{i}},W_\bullet\wt{\Vd_{i}})&\simeq (g(\wt{\Ud_{i+1}'}),W_\bullet(g(\wt{\Ud_{i+1}'})))\oplus (\Vd_i,W_\bullet \Vd_{i}),\\
        (\wt{\Ud_{i}},W_\bullet\wt{\Ud_{i}})&\simeq (f(\wt{\Vd_{i+1}'})(1),W_\bullet(f(\wt{\Vd_{i+1}'})(1)))\oplus (\Ud_i,W_\bullet \Ud_{i}),
\end{align*}
and analogous isomorphisms are induced by the other decompositions appearing in
(\ref{mar4-6}) and (\ref{ap1-2}).
In particular,
after regarding $\Vd_{i}$ (non-canonically) as a subspace of $\Vd$ as in Corollary~\ref{ap3-3-c},
the weight filtration of $\Vd_{i}$ as a mixed Hodge structure coincides with the filtration of $\Vd_{i}$ induced by $W_\bullet \Vd$.
Analogous statements hold for $\Ud_i$,
for the variants of $\Vd_i$ and $\Ud_i$ with tildes and primes,
and for their images under $f$ or $g$.
\end{lem}

\begin{proof}
This follows from \ref{may11-1} of Lemma~\ref{may8-1-1} and the strictness
of morphisms of mixed Hodge structures with respect to the weight filtrations.
\end{proof}

Lemma~\ref{may8-6} immediately gives the following corollary.

\begin{cor}\label{may12-1}
    We can choose the decomposition (\ref{ap8-6}) in Corollary~\ref{ap8-4} so that it is compatible with the weight filtrations.
More explicitly, 
as a diagram of filtered $\QQ$-vector spaces with respect to the weight filtrations,
the diagram (\ref{may8-4}) admits the following decomposition:
    \begin{equation}       
    \begin{aligned}
 &\bigoplus_{i\geq 0}\left(   \xymatrix{ \displaystyle\bigoplus_{r=0}^{i}N_{\Vd}^r(\Vd_i)(r)\ar@<0.5ex>[r]^-{f,\ (-1)}&
    \displaystyle\bigoplus_{r=0}^{i-1}fN_{\Vd}^r (\Vd_i)(r+1)
    \ar@<0.5ex>[l]^-{g}
    }\right)\\
\oplus  &\displaystyle\bigoplus_{i\geq 1}\left(   \xymatrix{ \displaystyle\bigoplus_{r= 0}^{i-1}N_{\Vd}^r(\Vd_i')(r)\ar@<0.5ex>[r]^-{f,\ (-1)}&
    \displaystyle\bigoplus_{r= 0}^{i-1}fN_{\Vd}^r (\Vd_i')(r+1)
    \ar@<0.5ex>[l]^-{g}
    }\right)\\
    \oplus 
&\bigoplus_{i\geq 0}   
\left(    \xymatrix{
 \displaystyle\bigoplus_{r=0}^{i-1}gN_\Ud^r(\Ud_i)(r)\ar@<0.5ex>[r]^-{f,\ (-1)}&
    \displaystyle\bigoplus_{r=0}^iN_\Ud^r (\Ud_i)(r)
    \ar@<0.5ex>[l]^-{g}}\right)\\
    \oplus 
    &\bigoplus_{i\geq 1}   
\left(    \xymatrix{\displaystyle
 \bigoplus_{r= 0}^{i-1}gN_\Ud^r(\Ud_i')(r)\ar@<0.5ex>[r]^-{f,\ (-1)}&
    \displaystyle\bigoplus_{r= 0}^{i-1}N_\Ud^r (\Ud_i')(r)
    \ar@<0.5ex>[l]^-{g}}\right).
\end{aligned}\label{may12-2} 
    \end{equation}

%Here, each term is equipped with the weight filtration, and we omit the symbols $W_\bullet$.

\end{cor}

In what follows, we fix such a decomposition (\ref{may12-2}).

\begin{lem}\label{may8-2}
    For the diagram (\ref{may8-4}),
    there exist isomorphisms $\eta_{\Vd}\colon \Vd \simto \GR^W_\bullet \Vd$
    and $\eta_{\Ud}\colon \Ud\simto \GR^W_\bullet \Ud$ as filtered $\QQ$-vector spaces,
    i.e. ignoring the Hodge filtrations,
    such that
$\GR^W_\bullet(\eta_\Vd)$ and
$\GR^W_\bullet(\eta_\Ud)$ are the identity maps and    
the following diagram commutes:
\[
\xymatrix{
\Ud\ar[r]^-g\ar[d]^-{\eta_{\Ud}}&\Vd\ar[r]^-f\ar[d]^-{\eta_{\Vd}}&\Ud(-1)\ar[d]^-{\eta_{\Ud}(-1)}\\
\GR^W_\bullet \Ud\ar[r]^-{\GR^W_\bullet g}&\GR^W_\bullet \Vd\ar[r]^-{\GR^W_\bullet f}&\GR^W_\bullet \Ud(-1).\\
}
\]
\end{lem}

\begin{proof}
    By \ref{may11-2} of Lemma~\ref{may8-1-1},
we choose an isomorphism 
\[\Vd_{i}\simeq \GR^W_\bullet \Vd_{i}\]
as filtered $\QQ$-vector spaces with respect to the weight filtrations, which induces the identity map on $\GR^W_\bullet \Vd_{i}$.
By the chosen decomposition~(\ref{may12-2}),
$f|_{\Vd_i}\colon \Vd_i\to f(\Vd_i)$ is either zero or an isomorphism.
Hence, in either case, we obtain a filtered isomorphism
\[f\Vd_{i}\simeq \GR^W_\bullet f\Vd_{i},\]
so that the following diagram commutes:
\[\xymatrix{
\Vd_{i}\ar[r]^f\ar[d]^{\simeq}&f\Vd_{i}\ar[d]^{\simeq}\\
\GR^W_\bullet \Vd_{i}\ar[r]^{\GR^W_\bullet f}&\GR^W_\bullet f\Vd_{i}.}\]
By repeating this procedure,
we obtain filtered isomorphisms
\[N^r_\Vd (\Vd_{i})\simeq \GR^W_\bullet N_{\Vd}^r(\Vd_{i}),\quad 
fN^r_\Vd (\Vd_{i})\simeq \GR^W_\bullet fN_{\Vd}^r(\Vd_{i}),
\]
so that the following diagram commutes:
\[
\xymatrix{
\bigoplus_r fN_{\Vd}^r(\Vd_i)(r+1) \ar[r]^-g\ar[d]^{\simeq }&
\bigoplus_r N_{\Vd}^r(\Vd_i)(r)
\ar[r]^-f\ar[d]^{\simeq }&\bigoplus_r fN_{\Vd}^r(\Vd_i)(r)\ar[d]^{\simeq }\\
\bigoplus_r \GR^W_\bullet fN_{\Vd}^r(\Vd_i)(r+1)\ar[r]^-{\GR^W_\bullet g}&\bigoplus_r\GR^W_\bullet  N_{\Vd}^r(\Vd_i)(r)
\ar[r]^-{\GR^W_\bullet f}&\bigoplus_r \GR^W_\bullet fN_{\Vd}^r(\Vd_i)(r).\\
}
\]
We note that the first row is a direct summand of the diagram: 
$\Ud\xrightarrow{g} \Vd\xrightarrow{f} \Ud(-1)$,
and the second row is a direct summand of the diagram:
$\GR^W_\bullet \Ud\xrightarrow{\GR^W_\bullet g} \GR^W_\bullet \Vd\xrightarrow{\GR^W_\bullet f} \GR^W_\bullet \Ud(-1)$.

%By Lemma~\ref{may8-6}, the direct sum decompositions
%$
%\bigoplus_\ell N_{\Vd}^\ell(\Vd_i)$ and
%$\bigoplus_\ell fN_{\Vd}^\ell(\Vd_i)$
%are compatible with the weight filtrations.
%Therefore, taking
%$\GR^W_\bullet$ commutes with these direct sums.

Applying the same argument to all the summands in the decomposition~(\ref{may12-2}), we obtain the desired result.    
\end{proof}

In what follows, we fix isomorphisms $\eta_{\Vd}$ and $\eta_{\Ud}$
satisfying the properties stated in Lemma~\ref{may8-2}.
After tensoring these isomorphisms with $\CC$, we transport the
Hodge filtrations of $\Vd$ and $\Ud$ to $\GR^W_\bullet \Vd$ and $\GR^W_\bullet \Ud$, respectively, and we will describe the resulting filtrations below.

We also fix a decomposition as in (\ref{ap10-2}) of Theorem~\ref{may15-5}. 

%For $M\in \MHM_{\mon}^{\unip}(\CC_t)$ which corresponds to the diagram~(\ref{mar3-3}),
%we fix a decomposition (\ref{ap10-2}).

\begin{defi}
For $1\leq s\leq \ell$,
        we define finite sets of integers $J_{\Vd}(s)$ and $J_{\Ud}(s)$ as
        \begin{align*}
            J_{\Vd}(s):=\left\{
        \begin{array}{cc}
             {[0,i(s)-1]}\cap \ZZ& \mbox{if\ } X^s=A, \\
             {[-1,i(s)-2]}\cap \ZZ& \mbox{if\ } X^s=B,\\
             {[0,i(s)-2]}\cap \ZZ&\mbox{if\ } X^s=C,\\
             {[-1,i(s)-1]}\cap \ZZ&\mbox{if\ } X^s=D,
        \end{array}
        \right.\quad
            J_{\Ud}(s):=[0,i(s)-1]\cap \ZZ,
        \end{align*}
        where $[a,b]\cap \ZZ=\emptyset$ whenever $a>b$.
\end{defi}

By definition, 
the $\Vd$-part and the $\Ud$-part of the diagram corresponding to $(X^s)_{i(s)}^H$
are
\begin{align*}
    \bigoplus_{j\in J_{\Vd}(s)}\QQ^H(j)\mbox{\quad and\quad}
    \bigoplus_{j\in J_{\Ud}(s)}\QQ^H(j),
\end{align*}
respectively.
Therefore,
the $\Vd$-part and the $\Ud$-part of
the diagram corresponding to 
$G(M)$
are
\begin{align}\label{ap15-2}
 \bigoplus_{s=1}^\ell  
 \bigoplus_{j\in J_{\Vd}(s)}H^s(j)\mbox{\quad and\quad}
 \bigoplus_{s=1}^\ell  \bigoplus_{j\in J_{\Ud}(s)}H^s(j),
\end{align}
respectively.
%Let $w_s$ be the weight of the pure Hodge structure $H^s$.
%Then, the weight of $H^s(j)$ is $w_s-2j$.

Thus,
after forgetting the Hodge filtrations,
we obtain the following chain of isomorphisms between diagrams of filtered $\QQ$-vector spaces:
\begin{equation}\label{may14-1}
      \begin{aligned}
    \xymatrix{
    \Vd\ar@<0.5ex>[r]^-{f,\ (-1)}\ar[d]^{\simeq }&\Ud\ar@<0.5ex>[l]^-{g}\ar[d]^{\simeq }\\
    \GR^W_\bullet \Vd\ar@<0.5ex>[r]^-{\GR^W_\bullet(f),\ (-1)}\ar[d]^{\simeq }&\GR^W_\bullet \Ud\ar@<0.5ex>[l]^-{\GR^W_\bullet(g)}\ar[d]^{\simeq }\\
    \bigoplus_{s=1}^\ell  
 \bigoplus_{j\in J_{\Vd}(s)}H^s(j)
 \ar@<0.5ex>[r]^-{\wt{f},\ (-1)}
 &
     \bigoplus_{s=1}^\ell  \bigoplus_{j\in J_{\Ud}(s)}H^s(j),    
  \ar@<0.5ex>[l]^-{\wt{g}} 
    }
\end{aligned}
\end{equation}
where the bottom arrows are the direct sums of the following maps:
\begin{align*}
    \wt{f}|_{H^s(j)}&=\left\{
    \begin{array}{cc}
         \iden_{H^s(j)}& j+1\in J_\Ud(s) \\
         0& j+1\notin J_\Ud(s)
    \end{array}
    \right.\quad \text{for\ } j\in J_\Vd(s),\quad \text{and}\\
        \wt{g}|_{H^s(j)}&=\left\{
    \begin{array}{cc}
         \iden_{H^s(j)}& j\in J_\Vd(s) \\
         0& j\notin J_\Vd(s)
    \end{array}
    \right.\quad \text{for\ } j\in J_\Ud(s).
    \end{align*}
In the first case in the definition of $\wt{f}|_{H^s(j)}$, we identify $H^s(j+1)(-1)$ with $H^s(j)$.

Let $\Lambda\subset \ZZ$ be the finite set consisting of the weights of $H^s(j)$ appearing in (\ref{ap15-2}).
We regard $\Lambda$ as an ordered set with the usual order.
Then, for $w\in \Lambda$, we define 
\begin{align}\label{may15-6}
    L_{\Vd,w}:=\bigoplus_{\substack{1\leq s\leq \ell\\
    j\in J_\Vd(s)\\
    \wei{H^s}-2j=w
    }}H^s(j),
    \quad\text{and}\quad
        L_{\Ud,w}:=\bigoplus_{\substack{1\leq s\leq \ell\\
    j\in J_\Ud(s)\\
    \wei{H^s}-2j=w
    }}H^s(j),
\end{align}
where $\wei{H^s}$ is the weight of $H^s$.
For $\taik=\QQ$ or $\CC$, we denote by $L_{\Vd,w,\taik}$ and $L_{\Ud,w,\taik}$
the underlying $\taik$-vector spaces of $L_{\Vd,w}$ and $L_{\Ud,w}$, respectively.
We set $L_{\Vd,w}=0$ and $L_{\Ud,w}=0$ for $w\notin \Lambda$.
Then,
via the fixed identifications (\ref{may14-1}),
$U$ and $V$ together with the induced identifications of their graded pieces,
determine elements
\[[\Vd]\in \Exten{\{L_{\Vd,w}\}_{w\in \Lambda}},\quad \text{and}\quad 
[\Ud]\in \Exten{\{L_{\Ud,w}\}_{w\in \Lambda}}.
\]
Thus, these mixed Hodge structures can be described in terms of upper
triangular matrices, as explained in Subsection~\ref{ssfixedgradedpieces}.
The morphisms
$\wt{f}$ and $\wt{g}$ in the diagram (\ref{may14-1})
induce morphisms of pure Hodge structures
\begin{align}\label{may15-7}
    f_w\colon L_{\Vd,w}\to L_{\Ud,w-2}(-1)\quad \text{and}\quad 
    g_w\colon L_{\Ud,w}\to L_{\Vd,w}.
\end{align}
Moreover, by the commutative diagram (\ref{may14-1}),
we can apply Proposition~\ref{sep28-1} 
to describe the condition that $f$ and $g$ are morphisms of mixed Hodge structures.

Let us summarize the argument above.

\begin{lem}\label{may18-1}
    For a diagram (\ref{mar3-3}),
we fix two isomorphisms of diagrams of filtered $\QQ$-vector spaces, ignoring the Hodge filtrations:
\begin{enumerate}
    \item  an isomorphism chosen as in Lemma~\ref{may8-2}:\[\xymatrix@C=50pt{
    \Vd\ar@<0.5ex>[r]^-{f,\ (-1)}\ar[d]^{\simeq }&\Ud\ar@<0.5ex>[l]^-{g}\ar[d]^{\simeq }\\
    \GR^W_\bullet \Vd\ar@<0.5ex>[r]^-{\GR^W_\bullet(f),\ (-1)}&\GR^W_\bullet \Ud,\ar@<0.5ex>[l]^-{\GR^W_\bullet(g)}}\]
    \item an isomorphism chosen as in Theorem~\ref{may15-5}:
    \[ \xymatrix{   \GR^W_\bullet \Vd\ar@<0.5ex>[r]^-{\GR^W_\bullet(f),\ (-1)}\ar[d]^{\simeq }&\GR^W_\bullet \Ud\ar@<0.5ex>[l]^-{\GR^W_\bullet(g)}\ar[d]^{\simeq }\\
    \bigoplus_{s=1}^\ell  
 \bigoplus_{j\in J_{\Vd}(s)}H^s(j)
 \ar@<0.5ex>[r]^-{\wt{f},\ (-1)}
 &
     \bigoplus_{s=1}^\ell  \bigoplus_{j\in J_{\Ud}(s)}H^s(j),
  \ar@<0.5ex>[l]^-{\wt{g}} 
    }\]
    \end{enumerate}
    and define $L_{\Vd,w}$, $L_{\Ud,w}$, $f_w$ and $g_w$  by (\ref{may15-6}) and (\ref{may15-7}).
    Here,
in the diagrams above, the subscript $\QQ$ for the underlying filtered $\QQ$-vector spaces is omitted.
    We set 
\begin{align*}
    \UpT_\Vd:=\UpT(\{L_{\Vd,w}\}_{w\in \Lambda}),\quad \text{and}\quad 
    \UpT_\Ud:=\UpT(\{L_{\Ud,w}\}_{w\in \Lambda}),
\end{align*}
and define $\UpT_{\Vd,\QQ}$, $\UpT_{\Ud,\QQ}$, $\UpT_{\Vd,F}$ and $\UpT_{\Ud,F}$
as in Definition~\ref{may14-3}.
Then, the mixed Hodge structures on $\Vd$ and $\Ud$,
together with the fixed identifications of their graded pieces,
are classified by
\[[M_{\Vd}]\in \UpT_{\Vd,\QQ}\backslash \UpT_\Vd/\UpT_{\Vd,F},\quad [M_\Ud]\in \UpT_{\Ud,\QQ}\backslash \UpT_\Ud/\UpT_{\Ud,F}.\]
By Proposition~\ref{sep28-1},
the condition that $f$ and $g$ are morphisms of mixed Hodge structures
is equivalent to the following condition:
we can choose representatives 
$M_{\Vd}=(m_{\Vd,ij})_{i,j\in \Lambda}\in \UpT_\Vd$ and $M_{\Ud}=(m_{\Ud,ij})_{ i,j\in \Lambda}\in \UpT_\Ud$
satisfying condition $(*)$ such that the diagram below commutes for every $i,j\in \Lambda$ with $i<j$.
\begin{equation}
    \label{may18-2}
    \begin{aligned}
\xymatrix@C=40pt{L_{\Ud,j,\CC}\ar[r]^{g_j}\ar[d]_{m_{\Ud,ij}}&L_{\Vd,j,\CC}\ar[r]^-{f_j}\ar[d]_{m_{\Vd,ij}}&L_{\Ud,j-2,\CC}(-1)\ar[d]^{m_{\Ud,(i-2)(j-2)}(-1)}\\L_{\Ud,i,\CC}\ar[r]_{g_i}&L_{\Vd,i,\CC}\ar[r]_-{f_i}&L_{\Ud,i-2,\CC}(-1).}
\end{aligned}
\end{equation}
Here, we set $m_{\Vd,ab}=0$ and $m_{\Ud,ab}=0$ whenever $a\notin \Lambda$ or $b\notin \Lambda$.
%\[f_im_{ij}=m_{(i-2)(j-2)}f_j,\quad \text{and}\quad g_im_{ij}=m_{ij}g_j.\]
\end{lem}

\subsection{Classification theorem}\label{ssclasmonmhm}

In this subsection, we give a description of the isomorphism classes of objects in $\MHM_{\mathrm{mon},\unip}(\CC_t)$.

\begin{defi}\label{may18-5}
    We define a category $\calC$ as follows.
    
    \noindent \underline{Objects}
An object of $\calC$ is a triple $(G,M_{\Vd},M_\Ud)$ satisfying the following conditions.

\begin{enumerate}
    \item $G$ is a split graded model as in Definition~\ref{may21-1}.
We write $G$ as the diagram
       \begin{align}\label{may27-7}
    \xymatrix@C=40pt{\Vd_G\ar@<0.5ex>[r]^-{f,\ (-1)}&\Ud_G\ar@<0.5ex>[l]^-{g}}
\end{align}

        \item 
      Let $\Vd_G=\bigoplus_{w\in \Lambda}L_{\Vd,w}$
        and
        $\Ud_G=\bigoplus_{w\in \Lambda}L_{\Ud,w}$
    be the weight decompositions of 
    the $\Vd$-part and the $\Ud$-part of $G$, respectively,
where $\Lambda$ is a finite set of weights. 
Then,
    \[
    M_{\Vd}=(m_{\Vd,ij})_{i,j\in \Lambda}\in \UpT(\{L_{\Vd,w}\}_{w\in\Lambda}),
    \quad\text{and}\quad 
    M_\Ud=(m_{\Ud,ij})_{i,j\in \Lambda}\in \UpT(\{L_{\Ud,w}\}_{w\in\Lambda}).\]
%    We require that both $M_{\Vd}$ and $M_\Ud$ satisfy condition $(*)$.

\item \label{may19-2}
Using the convention that $L_{\Vd,w}=0$ and $L_{\Ud,w}=0$ for $w\notin \Lambda$,
we define 
$f_w\colon L_{\Vd,w}\to L_{\Ud,w-2}(-1)$ and
$g_w\colon L_{\Ud,w}\to L_{\Vd,w}$ 
so that we have
$f=\bigoplus_{w\in \Lambda}f_w$ and $g=\bigoplus_{w\in \Lambda}g_w$.
Then, $M_{\Vd}$ and $M_\Ud$ are compatible with $f_w$ and $g_w$, i.e.
we have
\[
     f_{i}\circ m_{\Vd,ij}
    =
    m_{\Ud,(i-2)(j-2)}(-1)\circ f_{j},
    \qquad\text{and}\quad
      g_{i}\circ m_{\Ud,ij}
    =
    m_{\Vd,ij}\circ g_{j},
\]
 for any $i,j\in\Lambda$ with $i<j$.
 Here,
 we also set $m_{\Vd,ij}=0$ and $m_{\Ud,ij}=0$ if one of the indices is not in $\Lambda$.
\end{enumerate}
\underline{Morphisms}\  For such an object $(G,M_{\Vd},M_\Ud)$, let ${\Vd}_G^{M_{\Vd}}$ and ${\Ud}_G^{M_\Ud}$ be the mixed Hodge structures obtained from $\Vd_G$ and $\Ud_G$
by modifying their Hodge filtrations by $M_{\Vd}$ and $M_\Ud$, respectively.
Then, by condition \ref{may19-2} above, $f$ and $g$ induce
a diagram $D(G,M_{\Vd},M_\Ud)$ of mixed Hodge structures
  \begin{align}\label{may19-1}
    D(G,M_{\Vd},M_\Ud)\ \colon \ \xymatrix@C=40pt{{\Vd}^{M_{\Vd}}_G\ar@<0.5ex>[r]^-{f,\ (-1)}&{\Ud}^{M_\Ud}_G\ar@<0.5ex>[l]^-{g}}.
\end{align}
For two objects 
$(G,M_{\Vd},M_\Ud)$ and $(G',M'_{\Vd},M'_{\Ud})$,
a morphism $(G,M_{\Vd},M_\Ud)\to (G',M'_{\Vd},M'_{\Ud})$ is a 
morphism 
$D(G,M_{\Vd},M_\Ud)\to D(G',M'_{\Vd},M'_{\Ud})$ of diagrams of mixed Hodge structures.

\end{defi}

\begin{comment}
    Explicitly, a pair $(\alpha_V,\alpha_W)$ satisfying the following conditions.

\begin{enumerate}
    \item 
    $\alpha_V\colon \ov{\Vd_G}\to \ov{V_{G'}}$ and
$\alpha_W\colon \ov{\Ud_G}\to \ov{W_{G'}}$ are morphisms of mixed Hodge structures.
    \item $f_{G'}\circ \alpha_V=\alpha_W(-1)\circ f_G$ and $g_{G'}\circ \alpha_W=\alpha_V\circ g_{G}$.
\end{enumerate}

\end{comment}

\begin{prop}\label{aug4-1-prop}
There is a functor
\[T\colon \calC\to \MHM_{\mathrm{mon},\unip}(\CC_t),\]
which sends an object $(G,M_{\Vd},M_\Ud)\in \calC$ to the object in $\MHM_{\mathrm{mon},\unip}(\CC_t)$ corresponding to the diagram
$D(G,M_{\Vd},M_\Ud)$.
Moreover, $T$ is an equivalence of categories.
\end{prop}

\begin{proof}
Full faithfulness follows immediately from the definition of morphisms in $\calC$.
Essential surjectivity follows from Lemma~\ref{may18-1}.
\end{proof}

\begin{rem}
Condition $(*)$ is not imposed in Definition~\ref{may18-5}.
However, for an object $(G,M_V,M_U)\in \calC$,
applying Lemma~\ref{may18-1} to the diagram $D(G,M_V,M_U)$,
we can choose $M_V'$ and $M_U'$ satisfying condition $(*)$ such that 
$(G,M_V,M_U)$ and $(G,M_V',M_U')$ are isomorphic in $\calC$.
\end{rem}

\begin{defi}\label{may21-3}
 We use the notation in Definition~\ref{may18-5}.
\begin{enumerate}
\item \label{may20-6}
For a fixed split graded model $G$,
let $(M_{\Vd},M_\Ud), (M_{\Vd}',M_\Ud')\in \UpT(\{L_{\Vd,w}\}_{w\in\Lambda})\times \UpT(\{L_{\Ud,w}\}_{w\in\Lambda})$ be pairs such that
$(G,M_{\Vd},M_\Ud)$ and $(G,M_{\Vd}',M_\Ud')$ are objects of $\calC$.
We say that
these pairs are equivalent
if there exist
$A_{\Vd}\in \UpT_{\Vd,\QQ}$, $A_{\Ud}\in \UpT_{\Ud,\QQ}$, $B_{\Vd}\in \UpT_{\Vd,F}$ and $B_{\Ud}\in \UpT_{\Ud,F}$
such that
\[\xymatrix@C=40pt{
{\Vd}_{G}\ar[d]_{A_{\Vd}}^{\rotatebox{90}{$\sim$}}\ar@<0.5ex>[r]^-{f,\ (-1)}&{\Ud}_{
G}\ar@<0.5ex>[l]^-{g}\ar[d]_{\rotatebox{90}{$\sim$}}^{A_{\Ud}}
\\{\Vd}_{G}\ar@<0.5ex>[r]^-{f,\ (-1)}&{\Ud}_{
G}\ar@<0.5ex>[l]^-{g},}
\]
commutes and
\[M_{\Vd}'=A_{\Vd}M_{\Vd}B_{\Vd},\ANd M_\Ud'=A_{\Ud}M_\Ud B_{\Ud}.\]

\item  
For $G$ as above, we denote by $[M_{\Vd},M_\Ud]_G$ the equivalence class
of a pair $(M_{\Vd},M_\Ud)$.
We write $\frC_1(G)$ for the quotient set.
    
\item \label{may21-9}Let $G'$ be another split graded model,
with a diagram
     \begin{align*}
    \xymatrix@C=40pt{\Vd_{G'}\ar@<0.5ex>[r]^-{f',\ (-1)}&\Ud_{
G'}\ar@<0.5ex>[l]^-{g'}},
\end{align*}
with weight decompositions
$\Vd_{G'}=\bigoplus_{w\in \Lambda'}L_{\Vd,w}'$
        and
        $\Ud_{G'}=\bigoplus_{w\in \Lambda'}L_{\Ud,w}'$.
 Then, for $[M_{\Vd},M_\Ud]_G\in \frC_1(G)$
 and $[M'_\Vd,M'_\Ud]_{G'}\in \frC_1(G')$,
 we say that
 they are equivalent 
 if there exists an isomorphism $(a_\Vd,a_\Ud)$ of diagrams:
\[\xymatrix@C=40pt{
{\Vd}_{G}\ar[d]_{a_\Vd}^{\rotatebox{90}{$\sim$}}\ar@<0.5ex>[r]^-{f,\ (-1)}&{\Ud}_{G}\ar@<0.5ex>[l]^-{g}\ar[d]_{\rotatebox{90}{$\sim$}}^{a_{\Ud}}
\\{\Vd}_{G'}\ar@<0.5ex>[r]^-{f',\ (-1)}&{\Ud}_{
G'}\ar@<0.5ex>[l]^-{g'},}
\]
satisfying the following properties:
\begin{enumerate}
    \item $a_\Vd\colon \Vd_G\simto \Vd_{G'}$ and $a_\Ud\colon \Ud_G\simto \Ud_{G'}$ decompose as direct sums of morphisms of  pure Hodge structures:
 \[a_\Vd=\bigoplus_{w\in \ZZ} a_{\Vd,w}\colon  \bigoplus_{w\in\Lambda}L_{\Vd,w}\simto \bigoplus_{w\in\Lambda'}L_{\Vd,w}',\ANd
 a_\Ud=\bigoplus_{w\in \ZZ} a_{\Ud,w}\colon  \bigoplus_{w\in\Lambda}L_{\Ud,w}\simto \bigoplus_{w\in\Lambda'}L_{\Ud,w}',
 \]      
 where $L_{\Vd,w}=0$ and $L_{\Ud,w}=0$ (resp. $L_{\Vd,w}'=0$ and $L_{\Ud,w}'=0$) for $w\notin \Lambda$ (resp. $w\notin \Lambda'$).
 \item \label{may21-8} Regarding $M_{\Vd}$ and $M_\Ud$ as endomorphisms,
 we have $[a_\Vd M_{\Vd}a_\Vd^{-1},a_\Ud M_\Ud a_\Ud^{-1}]_{G'}=[M'_\Vd,M'_{\Ud}]_{G'}$.
\end{enumerate}

\item 
The relation defined in \ref{may21-9} gives an
equivalence relation on $\bigsqcup_G\frC_1(G)$.
We define $\frC_2$ to be the
quotient set of $\bigsqcup_G\frC_1(G)$ by this equivalence relation.
For
$[M_{\Vd},M_\Ud]_G\in \bigsqcup_G\frC_1(G)$, we denote by
$[[M_{\Vd},M_\Ud]_G]$ its equivalence class. 

        \end{enumerate}
\end{defi}

\begin{rem}
In the setting of \ref{may20-6} of Definition~\ref{may21-3},
the pairs $(M_{\Vd},M_\Ud)$ and $(M_{\Vd}',M_\Ud')$ are equivalent
if and only if 
the diagrams $D(G,M_{\Vd},M_\Ud)$ and $D(G,M'_\Vd,M'_\Ud)$ of mixed Hodge structures
are isomorphic without changing the fixed split graded model $G$.
More precisely,
there exists an isomorphism of diagrams of mixed Hodge structures:
\[\xymatrix@C=40pt{
{\Vd}^{M_{\Vd}}_{G}\ar[d]_{A_{\Vd}}^{\rotatebox{90}{$\sim$}}\ar@<0.5ex>[r]^-{f,\ (-1)}&{\Ud}^{M_\Ud}_{
G}\ar@<0.5ex>[l]^-{g}\ar[d]_{\rotatebox{90}{$\sim$}}^{A_{\Ud}}
\\{\Vd}_{G}^{M_{\Vd}'}\ar@<0.5ex>[r]^-{f,\ (-1)}&{\Ud}^{M'_\Ud}_{
G}\ar@<0.5ex>[l]^-{g},}
\]
such that $\GR^W_\bullet A_{\Vd}$ and $\GR^W_\bullet A_{\Ud}$ are the identity maps.
On the other hand,
the equivalence relation between
$[M_{\Vd},M_\Ud]_G$ and $[M'_\Vd,M'_\Ud]_{G'}$ defined in \ref{may21-9} of Definition~\ref{may21-3}
means that 
after transporting $(M_{\Vd},M_\Ud)$ along an isomorphism of diagrams $G\to G'$,
the resulting pair is equivalent to $(M'_\Vd,M'_\Ud)$ in the sense of \ref{may20-6} of Definition~\ref{may21-3}.
\end{rem}

\begin{prop}\label{may27-9}
    Let 
 $\calC/\simeq$ and
 $\MHM_{\mathrm{mon},\unip}(\CC_t)/\simeq$ be the sets of isomorphism classes of objects in $\calC$ and $\MHM_{\mathrm{mon},\unip}(\CC_t)$, respectively.
 Then,
 there exists a bijection
 \[\frC_2\simto \calC/\simeq .\]
 Consequently, via the equivalence of categories in Proposition~\ref{aug4-1-prop}, 
 we have a bijection
 \[\frC_2\simto \MHM_{\mathrm{mon},\unip}(\CC_t)/\simeq.\]
\end{prop}

\begin{proof}
For a fixed $G$,
we have a natural map $\frC_1(G)\to \calC/\simeq$,
which sends $[M_{\Vd},M_\Ud]_G$ to the isomorphism class of $(G,M_{\Vd},M_\Ud)\in \calC$.
By the definition of the equivalence relation in \ref{may20-6},
this map is well-defined.
Moreover, if
$[M_{\Vd},M_\Ud]_G$ and $[M'_\Vd,M'_\Ud]_{G'}$ are equivalent in $\bigsqcup_G\frC_1(G)$,
then
$(G,M_{\Vd},M_\Ud)$ and $(G',M'_\Vd,M'_\Ud)$ are also isomorphic in $\calC$.
Therefore,
we obtain a map $\frC_2\to \calC/\simeq$.
The surjectivity is clear.
Let us show the injectivity.
For $[[M_{\Vd},M_\Ud]_G]$ and $[[M'_\Vd,M'_\Ud]_{G'}]$ in $\frC_2$,
suppose that $(G,M_{\Vd},M_\Ud)$ and $(G',M'_\Vd,M'_\Ud)$ are isomorphic in $\calC$.
This means that
we have isomorphisms 
$P_\Vd\colon {\Vd}^{M_{\Vd}}_{G}\to {\Vd}^{M'_\Vd}_{G'}$ and
$P_\Ud\colon {\Ud}^{M_\Ud}_{G}\to {\Ud}^{M'_\Ud}_{G'}$,
which define an isomorphism between the diagrams
$D(G,M_{\Vd},M_\Ud)$ and $D(G',M'_\Vd,M'_\Ud)$.
Let $\GR^W_\bullet P_\Vd$ and $\GR^W_\bullet P_\Ud$ be the associated graded isomorphisms.
Moreover,
we define upper triangular matrices in $\UpT_\Vd$ and $\UpT_\Ud$
\[M''_\Vd:=(\GR^W_\bullet P_\Vd)^{-1}M_{\Vd}'\GR^W_\bullet P_\Vd,\qquad M''_\Ud:=(\GR^W_\bullet P_\Ud)^{-1}M_\Ud'\GR^W_\bullet P_\Ud,
\]
respectively.
Set $A_{\Vd}:=(\GR^W_\bullet P_{\Vd})^{-1}P_{\Vd}$ and
$A_{\Ud}:=(\GR^W_\bullet P_{\Ud})^{-1}P_{\Ud}$.
Then, we obtain the following commutative diagram:

\[\xymatrix@C=40pt{
&{\Vd}^{M_{\Vd}}_{G}\ar[d]_{P_{\Vd}}^{\rotatebox{90}{$\sim$}}\ar@<0.5ex>[r]^-{f,\ (-1)}
 \ar `l[ld]`[dd]_-{A_{\Vd}}[dd]
&\ar `r[rd]`[dd]^-{A_{\Ud}}[dd]{\Ud}^{M_\Ud}_{
G}\ar@<0.5ex>[l]^-{g}\ar[d]_{\rotatebox{90}{$\sim$}}^{P_{\Ud}}&
\\&{\Vd}_{G'}^{M_{\Vd}'}\ar@<0.5ex>[r]^-{f',\ (-1)}&{\Ud}^{M'_{\Ud}}_{
G'}\ar@<0.5ex>[l]^-{g'}&
\\&{\Vd}_{G}^{M''_{\Vd}}\ar@<0.5ex>[r]^-{f,\ (-1)}
\ar[u]_{\rotatebox{90}{$\sim$}}^{\GR^W_\bullet P_{\Vd}}
&{\Ud}^{M''_{\Ud}}_{
G}\ar@<0.5ex>[l]^-{g}
\ar[u]^{\rotatebox{90}{$\sim$}}_{\GR^W_\bullet P_{\Ud}}
.&}
\]
Since 
$\GR^W_\bullet A_{\Vd}$ and $\GR^W_\bullet A_{\Ud}$ are the identity maps,
$A_{\Vd}\in \UpT_{\Vd,\QQ}$ and $A_{\Ud}\in \UpT_{\Ud,\QQ}$.
Moreover,
since $A_{\Vd}^{-1}$ and $A_{\Ud}^{-1}$ are morphisms of mixed Hodge structures,
$B_{\Vd}:=M_{\Vd}^{-1}A_{\Vd}^{-1}M''_{\Vd}$ and $B_{\Ud}:=M_\Ud^{-1}A_{\Ud}^{-1}M''_{\Ud}$
are in $\UpT_{\Vd,F}$ and $\UpT_{\Ud,F}$, respectively.
Then,
the equalities $A_{\Vd}M_{\Vd}B_{\Vd}=M_{\Vd}''$ and $A_{\Ud}M_\Ud B_{\Ud}=M''_{\Ud}$
imply
that
the pairs $(M_{\Vd},M_\Ud)$ and $(M''_{\Vd},M''_{\Ud})$ are equivalent in the sense of
\ref{may20-6}.
Hence, $[M_{\Vd},M_\Ud]_G=[M''_{\Vd},M''_{\Ud}]_G$ in $\frC_1(G)$.
Furthermore,
by the definition of the equivalence relation in
\ref{may21-9},
$[M'_{\Vd},M'_{\Ud}]_{G'}$ and $[M''_{\Vd},M''_{\Ud}]_{G}$ are equivalent in $\bigsqcup_{G''}\frC_1(G'')$.
Hence,
we obtain 
$[[M_{\Vd},M_\Ud]_G]=[[M'_{\Vd},M'_{\Ud}]_{G'}]$ in $\frC_2$.
This proves the injectivity of $\frC_{2}\to \calC/\simeq$,
and we thus obtain the bijection $\frC_{2}\simto \calC/\simeq$.
\end{proof}

To describe the isomorphism classes of objects in $\calC$, let us consider the constraints on $M_{\Vd}$ and $M_\Ud$ for an object of $\calC$.
The next lemma follows immediately from the definition.

\begin{lem}\label{may30-1}
    For $(G,M_{\Vd},M_\Ud)\in \calC$,
    write $G=\bigoplus_{s=1}^\ell (X^s)^H_{i(s)}\otimes H^s$ as in (\ref{ap10-2})
    and 
let $L_{\Vd,w}$ and $L_{\Ud,w}$ be defined as in (\ref{may15-6}).
Let $w_s$ denote the weight of $H^s$.
For $1\leq s\leq \ell$, $j\in \ZZ$ and $Z\in \{\Vd,\Ud\}$, we set 
\[H_Z(s,j)=\begin{cases}
    H^s(j)&j\in J_Z(s)\\
    0&j\notin J_Z(s),
\end{cases}\]
and regard it as a direct summand of $L_{Z,w_s-2j}$ if $j\in J_Z(s)$,
and as the zero summand otherwise.
Then, the morphism
$H_{\Vd}(s,p)\to H_{\Ud}(s,p+1)(-1)$
(resp. $H_{\Ud}(s,p)\to H_{\Vd}(s,p)$)
induced by $f$ (resp. $g$)
is 
\begin{align*}
\begin{cases}
    \iden & p\in J_{\Vd}(s),\ p+1\in J_{\Ud}(s)\\
    0& \text{otherwise}
\end{cases}\quad 
\left(resp.\ 
\begin{cases}
    \iden & p\in J_{\Ud}(s),\ p\in J_{\Vd}(s)\\
    0& \text{otherwise}
\end{cases}
\right).
\end{align*}
In the first case where the map is the identity, we identify $H^s(p+1)(-1)$ with $H^s(p)$.
Moreover, for $1\leq s,t\leq \ell$ and $p,q\in \ZZ$,
by using the above description of $f$ and $g$,
we obtain the following commutative diagram:
\begin{equation}
    \label{may30-2}
\begin{aligned}
    \xymatrix@C=40pt{H_{\Ud}(s,p)\ar[r]^{g}\ar[d]&H_{\Vd}(s,p)\ar[r]^-{f}\ar[d]&H_{\Ud}(s,p+1)(-1)\ar[d]
\\
H_{\Ud}(t,q)\ar[r]^{g}&H_{\Vd}(t,q)\ar[r]^-{f}&H_{\Ud}(t,q+1)(-1).
}
\end{aligned}
\end{equation}
Here, the vertical arrows are the corresponding components of the vertical arrows in
\[
\xymatrix@C=40pt{
L_{\Ud,w}\ar[r]^{g_w}\ar[d]_{m_{\Ud,w'w}}&L_{\Vd,w}\ar[r]^-{f_{w}}\ar[d]^{m_{\Vd,w'w}}&L_{\Ud,w-2}(-1)\ar[d]^{m_{\Ud,(w'-2)(w-2)}(-1)}\\
L_{\Ud,w'}\ar[r]^{g_{w'}}&L_{\Vd,w'}\ar[r]^-{f_{w'}}&L_{\Ud,w'-2}(-1),
}
\]
where $w=w_s-2p$ and $w'=w_t-2q$.

\end{lem}

For $G$,
we set
\begin{align}\label{aug5-9}
    Z_s=\left\{
\begin{array}{cc}
     \Vd& \mbox{if}\  X^s\in \{B,D\}, \\
     \Ud& \mbox{if}\ X^s\in\{A,C\},
\end{array}\right.
&\quad
Z'_s=\left\{
\begin{array}{cc}
     \Vd& \mbox{if}\ X^s\in \{B,C\}, \\
     \Ud& \mbox{if}\ X^s\in \{A,D\},
\end{array}\right.
\end{align}
\begin{align}\label{oct8-1}
    a_s:=\begin{cases}
    0&\text{if}\ X^s\in\{A,C\},\\
    -1&\text{if}\ X^s\in\{B,D\}.
\end{cases}
\end{align}
Thus, $Z_s$ (resp. $Z'_s$) denotes the side from which the diagram starts (resp. the side on which it reaches $0$ first),
and $a_s$ is the integer such that the first term of the diagram of $(X^s)^H_{i(s)}$ is $\QQ^H(a_s)$.

For $1\leq s,t\leq \ell$, $p,q\in \ZZ$ and $Y\in \{\Vd,\Ud\}$,
    we denote by $\nu_{Y,s,t,p,q}\colon H_Y(s,p)\to H_{Y}(t,q)$
    the corresponding vertical arrow in the diagram (\ref{may30-2}).
    By repeatedly using the commutativity of (\ref{may30-2}) with Lemma~\ref{oct6-1} below,
we have the following lemma, which tells us the constraints on $M_{\Vd}$ and $M_\Ud$ for an object in $\calC$.

\begin{lem}\label{may30-5-1}
    In the situation of Lemma~\ref{may30-1},
    we have the following.
\[\nu_{Y,s,t,p,q}=\begin{cases}
    \nu_{Z_s,s,t,a_s,a_s+q-p}(p-a_s)&\text{if}\quad  a_s+q-p\in J_{Z_s}(t),\  p\in J_Y(s),\  q\in J_{Y}(t),\\
    0&\text{otherwise}.
\end{cases}\]
Moreover, for $q\in J_{Z_s}(t)$, if there exists $Y\in \{V,U\}$ and $k\geq -a_s$ such that $a_s+k\notin J_{Y}(s)$ and $q+k\in J_{Y}(t)$,
then for any $m\in \ZZ$ we have
\[\nu_{Z_s,s,t,a_s+m,q+m}=0.\]
\end{lem}

\begin{comment}
    \begin{enumerate}
    \item \label{aug4-2} If $q-p\notin J_{Z_s}(t)$, then $\nu_{Y,s,t,p,q}=0$.
    \item \label{aug4-3} If $q-p\in J_{Z_s}(t)$ and $p\in J_{Y}(s)$ and $q\in J_{Y}(t)$,
    then $\nu_{Y,s,t,p,q}=\nu_{Z_s,s,t,0,q-p}(p)$.
    In particular, if $q<p$, since $q-p\notin J_{Z_s}(t)$, we have $\nu_{Y,s,t,p,q}=0$.
    \item \label{aug4-4} For $q\in J_{Z_s}(t)$, if there exists $k\in \ZZ_{\geq 0}$ such that $k\notin J_{Y}(s)$ and $q+k\in J_{Y}(t)$,
    then $\nu_{Z_s,s,t,0,q}=0$.
\end{enumerate}
\end{comment}

\begin{lem}\label{oct6-1}
        In the situation of Lemma~\ref{may30-1},
let 
\[\xymatrix{
I(s)\ar[d]_{m}\ar[r]&I'(s)\ar[d]^{m'}\\
I(t)\ar[r]&I'(t)
}\]
be either the left or right square of (\ref{may30-2}).
Then, we have the following.
\begin{enumerate}
    \item If $I(s)$, $I'(s)$, $I(t)$ and $I'(t)$ are not zero, then $m$ and $m'$ are identified.
    \item If $I(t)=0$ and $I(s)\neq 0$, then $m$ and $m'$ are zero.
    \item If $I'(s)=0$ and $I'(t)\neq 0$, then $m$ and $m'$ are zero.
\end{enumerate}
\end{lem}

\begin{defilem}
Let $G$ be a split graded model.
We use the decomposition~(\ref{may15-6}).
    \begin{enumerate}
        \item For $1\leq s\leq \ell$ and $Y\in \{\Ud,\Vd\}$,
        we define an integer $\wtmax J_{Y}(s)\in \ZZ$ as
\[
\wtmax J_{Y}(s):=\left\{\begin{array}{cl}
             \max J_{Y}(s)& (J_{Y}(s)\neq \emptyset) \\
             -1& (J_{Y}(s)=\emptyset).
        \end{array}\right.
       \] 
Then, for any $s$ we have
\begin{equation}\label{may27-2}
\wtmax J_{\Ud}(s)-1\leq \wtmax J_{\Vd}(s)\leq \wtmax J_{\Ud}(s).
\end{equation}

\item \label{may27-1}
For $1\leq s,t\leq \ell$ and $q\in J_{Z_s}(t)$,
the following conditions are equivalent:
\begin{enumerate}[label=(\alph*), ref=\alph*]
    \item \label{may26-1}
$\wtmax{J_{\Vd}(s)}\geq \wtmax{J_{\Vd}(t)}-q+a_s $ and\ \   $\wtmax{J_{\Ud}(s)}\geq \wtmax{J_{\Ud}(t)}-q+a_s$.    
 \item \label{may26-3}
$\wtmax{J_{Z'_s}(s)}\geq \wtmax{J_{Z'_s}(t)}-q+a_s$.
    \item\label{may26-4}
For $k\geq -a_s$ and $Y\in \{\Vd,\Ud\}$, we have
\[a_s+k\notin J_{Y}(s) \Rightarrow q+k\notin J_{Y}(t).\]
    \item\label{may26-5}
For $k\geq -a_s$, we have
\[a_s+k\notin J_{Z'_s}(s) \Rightarrow q+k\notin J_{Z'_s}(t).\]
\end{enumerate}
We say that $(s,t,q)$ satisfies condition ($\delta$) if these equivalent conditions hold.   
    \end{enumerate}
\end{defilem}

\begin{proof}
By the definition,
we have
\begin{align*}
    \wtmax{J_U(s)}&=i(s)-1\\
    \wtmax{J_V(s)}&=\begin{cases}
        i(s)-1&\text{if}\ Z_s'=U,\ \text{i.e.}\ 
        X^s\in\{A,D\}\\
        i(s)-2&\text{if}\ Z_s'=V,\ \text{i.e.}\ X^s\in\{B,C\}.
    \end{cases}
\end{align*}
Hence, 
if  $X^s\in\{A,D\}$,
we have  $\wtmax{J_U(s)}=\wtmax{J_V(s)}$,
and if $X^s\in\{B,C\}$,
we have 
$\wtmax{J_U(s)}=\wtmax{J_V(s)}+1$.
This proves the inequality (\ref{may27-2}).

For the proof of \ref{may27-1}, 
assume $q\in J_{Z_s}(t)$.
By the definitions of $a_s$ and $Z_s$, if $k\geq -a_s$, then $a_s+k\geq 0$ and $q+k\geq 0$.
Note that $J_Y(u)\cap \ZZ_{\geq 0}=[0,\wtmax{J_Y}(u)]\cap \ZZ$.
Therefore, 
the condition 
$\wtmax{J_Y(s)}\geq \wtmax{J_Y(t)}-q+a_s$ is equivalent to
the condition 
$a_s+k\notin J_Y(s)\Rightarrow q+k\notin J_Y(t)$ for $k\geq -a_s$.
Hence, we obtain (\ref{may26-1})$\Leftrightarrow$ (\ref{may26-4}) and
(\ref{may26-3})$\Leftrightarrow$ (\ref{may26-5}).

It remains to show the equivalence between (\ref{may26-1}) and (\ref{may26-3}).
This follows immediately from the formulas for
$\wtmax{J_V(s)}$ and $\wtmax{J_U(s)}$ above
together with the inequality (\ref{may27-2}).

\end{proof}

\begin{defi}\label{may22-1}
     We define $\frC_3$ to be the set of pairs $(G,\{\mu_{s,t,q}\}_{(s,t,q)\in \Delta})$,
   where 
   $G$ is a split graded model and
    \[\mu_{s,t,q}\in \Hom_{\CC}(H^s(a_s),H^{t}(q)),\]
    for every $(s,t,q)\in \Delta$,
where the index set $\Delta$ is defined as
\[\Delta=\left\{(s,t,q)\in ([1,\ell]\cap\ZZ)^2\times \ZZ\ \middle|\ 
\begin{array}{l}
      q\in J_{Z_s}(t),\\
      \text{$(s,t,q)$ satisfies condition ($\delta$)},\\
      w_s-2a_s>w_t-2q
\end{array} 
\right\}.\]
\end{defi}

\begin{rem}
$H^s(a_s)$ is the first term of the diagram of $(X^s)^H_{i(s)}\otimes H^s$.
The condition $q\in J_{Z_s}(t)$ means that 
$H^t(q)$ is a direct summand of the $Z_s$-part with respect to the decomposition (\ref{may15-6}).
    The condition $w_s-2a_s>w_t-2q$ in Definition~\ref{may22-1}
    means that the weight of $H^s(a_s)$ is greater than that of $H^t(q)$.
Under these conditions,
condition ($\delta$)
is imposed in order to make the maps $\mu_{s,t,q}(k)$, which appear below as components of $M_{\Vd}$ and $M_{\Ud}$, compatible with $f$ and $g$.
\end{rem}

\begin{lem}\label{may27-8}
    For $(G,\mu)=(G,\{\mu_{s,t,q}\}_{(s,t,q)\in \Delta})\in \frC_3$,
there exists a unique pair $(M_{\Vd},M_\Ud)$ with $M_{\Vd}=(m_{\Vd,ij})_{i,j\in \Lambda}\in \UpT_\Vd$ and $M_\Ud=(m_{\Ud,ij})_{i,j\in \Lambda}\in \UpT_\Ud$ satisfying the following conditions:
\begin{enumerate}
    \item $(G,M_{\Vd},M_\Ud)$ is an object in $\calC$.
\item With respect to the decomposition~(\ref{may15-6}),
for $(s,t,q)\in \Delta$,
the component $H^s(a_s)\to H^t(q)$ of $m_{Z_s,(w_t-2q)(w_s-2a_s)}$ is $\mu_{s,t,q}$.
\end{enumerate}
    
\end{lem}

\begin{proof}
We use the notation defined in Lemma~\ref{may30-1} and Lemma~\ref{may30-5-1}.
For $Y\in\{\Vd,\Ud\}$ and $i,j\in \Lambda$, we set $m_{Y,ii}=\iden$ and 
$m_{Y,ij}=0$ if $i>j$.
It remains to define $m_{Y,ij}$ for $i<j$.
For $1\leq s,t\leq \ell$, $q\in \ZZ$, $k\in \ZZ_{\geq 0}$ and $a_s+k\in J_{Y}(s)$ such that $q+k\in J_{Y}(t)$
and $w_t-2(q+k)<w_s-2(a_s+k)$,
we define the component $H^s(a_s+k)\to H^t(q+k)$ of $m_{Y,(w_t-2(q+k))(w_s-2(a_s+k))}\colon L_{Y,w_s-2(a_s+k)}\to L_{Y,w_t-2(q+k)}$ with respect to
the decomposition~(\ref{may15-6})
to be 
\begin{align}\label{may28-1}
    \begin{cases}
    \mu_{s,t,q}(k)&\text{if}\  (s,t,q)\in \Delta,\\
    0&\text{otherwise},
\end{cases}
\end{align}
where $\mu_{s,t,q}(k)$ is the $k$-th Tate twist of
$\mu_{s,t,q}\colon H^s(a_s)\longrightarrow H^t(q)$.
We set all the remaining components to zero. 
We thus obtain $M_{\Vd}\in \UpT_\Vd$ and $M_\Ud\in \UpT_\Ud$.

Using the description of $f$ in Lemma~\ref{may30-1},
the equality $f\circ M_{\Vd}=M_\Ud(-1)\circ f$ holds if and only if
for $1\leq s,t\leq \ell$ and $p,q\in \ZZ$,
the following diagram commutes:
\begin{equation}
    \label{may25-2}
\begin{aligned}
    \xymatrix{
H_{\Vd}(s,p)\ar[r]^-{f}\ar[d]_{M_{\Vd}}&H_{\Ud}(s,p+1)(-1)\ar[d]^{M_\Ud(-1)}\\
H_\Vd(t,q)\ar[r]^-{f}&H_{\Ud}(t,q+1)(-1).
}
\end{aligned}
\end{equation}
Let us show the commutativity of (\ref{may25-2}).
In the case where $p> q$ or $w_s-2p< w_t-2q$,
both the vertical arrows are zero by construction.
If $w_s-2p=w_t-2q$, the commutativity follows from the fact that $m_{Y,ii}=\iden$.
Moreover, the case where $p\not\in J_{\Vd}(s)$, $q\not\in J_{\Vd}(t)$, or $q+1\not\in J_{\Ud}(t)$ is also trivial.
So we may assume that
$p\leq  q$, $w_s-2p>w_t-2q$, $p\in J_{\Vd}(s)$, $q\in J_{\Vd}(t)$, and $q+1\in J_{\Ud}(t)$.
Set $r=a_s+q-p$.
Then, the bottom arrow of (\ref{may25-2}) is the identity.
If $p+1\in J_{\Ud}(s)$ and $(s,t,r)\in \Delta$,
since the top arrow of (\ref{may25-2}) is the identity
and the vertical arrows are $\mu_{s,t,r}(p-a_s)$ and $\mu_{s,t,r}(p-a_s+1)(-1)$, respectively, 
the diagram commutes.
If $p+1\in J_{\Ud}(s)$ and $(s,t,r)\notin \Delta$,
both vertical arrows are zero by construction.
Suppose $p+1\not\in J_{\Ud}(s)$.
If $r\not\in J_{Z_s}(t)$, both vertical arrows are zero by construction,
so we may assume $r\in J_{Z_s}(t)$.
Combining the assumption $q+1\in J_{\Ud}(t)$,
we have $\wtmax J_{\Ud}(s)\leq p$ and $\wtmax J_{\Ud}(t)-r+a_s\geq q+1-r+a_s=p+1>p$,
which means that condition ($\delta$) does not hold.
Hence, $(s,t,r)$ is not in $\Delta$ and the vertical arrows are also zero in this case.
Similarly, we obtain $g\circ M_\Ud=M_{\Vd}\circ g$.

Conversely,
suppose that we have such a pair $(M_{\Vd},M_\Ud)$.
For $1\leq s,t\leq \ell$, $p,q\in \ZZ$, and $Y\in \{\Vd,\Ud\}$,
let
$\nu_{Y,s,t,p,q}\colon H_Y(s,p)\to H_Y(t,q)$
be the component of the morphism
$L_{Y,w_s-2p}\to L_{Y,w_t-2q}$
induced by $m_{Y,(w_t-2q)(w_s-2p)}$.
If $p\notin J_Y(s)$ or $q\notin J_{Y}(t)$,
this component is zero.
Hence, we assume that $p\in J_Y(s)$ and $q\in J_Y(t)$, and set $r=a_s+q-p$.

If $r\notin J_{Z_s}(t)$, then $\nu_{Y,s,t,p,q}=0$ by Lemma~\ref{may30-5-1}.
If $r\in J_{Z_s}(t)$, then Lemma~\ref{may30-5-1} gives
$\nu_{Y,s,t,p,q}=\nu_{Z_s,s,t,a_s,r}(p-a_s)$.
Thus, it remains to determine $\nu_{Z_s,s,t,a_s,r}$.

If $(s,t,r)$ does not satisfy condition ($\delta$),
then $\nu_{Z_s,s,t,a_s,r}=0$ by Lemma~\ref{may30-5-1}.
Suppose that condition ($\delta$) holds for $(s,t,r)$.
If $w_s-2a_s>w_t-2r$, then $(s,t,r)\in \Delta$, and the second condition in the statement of this Lemma gives $\nu_{Z_s,s,t,a_s,r}=\mu_{s,t,r}$.
If $w_s-2a_s<w_t-2r$, then this component is zero since $M_{\Vd}$ and $M_{\Ud}$ are upper triangular.
Finally, if $w_s-2a_s=w_t-2r$, the component is uniquely determined by $m_{Z_s,(w_s-2a_s)(w_s-2a_s)}=\iden$.
Therefore, all components $\nu_{Y,s,t,p,q}$ are uniquely determined by $\{\mu_{s,t,q}\}_{(s,t,q)\in \Delta}$.

This completes the proof.
\begin{comment}
By the assumption on $M_{\Vd}$ and $M_\Ud$,
for $(s,t,q)\in \Delta$, we have 
\[\nu_{Z_s,s,t,0,q}=\mu_{s,t,q}.\]
By (\ref{aug4-2}) of Lemma~\ref{may30-5-1},
we already know $\nu_{Z,s,t,p,q}=0$ if $q-p\notin J_{Z_s}(t)$,
Moreover, (\ref{aug4-2}) of Lemma~\ref{may30-5-1},
if $q-p\in J_{Z_s}(t)$, $p\in J_{Z}(s)$, $q\in J_Z(t)$ and $(s,t,q-p)\in \Delta$,
then we obtain $\nu_{Z,s,t,p,q}=\nu_{Z_s,s,t,0,q-p}(p)=\mu_{s,t,q-p}(p)$.
So, it remains to consider 
$\nu_{Z,s,t,p,q}$
for $p\in J_{Z}(s)$, $q\in J_Z(t)$
satisfying $q-p\in J_{Z_s}(t)$ and $(s,t,q-p)\notin \Delta$.
Suppose $(s,t,q-p)$ does not satisfy condition $(\delta)$.
Then, by (\ref{aug4-3}) of Lemma~\ref{may30-5-1},
$\nu_{Z_s,s,t,0,q-p}=0$, which implies $\nu_{Z,s,t,p,q}=0$.
Suppose that $w_s<w_t-2(q-p)$.
Since $\nu_{Z,s,t,p,q}\colon H^s(p)\to H^t(q)$ is induced from 
$L_{Z,w_s-2p}\to L_{Z,w_t-2q}$.
However, since $M_{\Vd}$ and $M_\Ud$ are upper triangular,
the inequality $w_s-2p<w_t-2q$ implies the map is zero,
which leads to $\nu_{Z,s,t,p,q}=0$.
Therefore,
we conclude that all maps $\nu_{Z,s,t,p,q}$ are uniquely determined 
from the maps $\{\mu_{s,t,q}\}_{(s,t,q)\in \Delta}$.

This completes the proof.
\end{comment}

\end{proof}

\begin{thm}\label{aug5-5}
There is a natural surjective map
\[\frC_3\twoheadrightarrow \frC_2,\]
which sends $(G,\{\mu_{s,t,q}\}_{(s,t,q)\in \Delta})$
to $[[M_{\Vd},M_\Ud]_G]$,
where $M_{\Vd}$ and $M_\Ud$ are constructed as in Lemma~\ref{may27-8}.
Consequently, via the bijection
$\frC_2\simto  \calC/\simeq$ in Proposition~\ref{may27-9}
and the equivalence of categories in Proposition~\ref{aug4-1-prop},
we obtain a surjective map
\[\frC_3 \twoheadrightarrow \MHM_{\mathrm{mon},\unip}(\CC_t)/\simeq.\]
In particular,
every object of $\MHM_{\mathrm{mon},\unip}(\CC_t)$ is
represented by a split graded model 
together with a family of $\CC$-linear maps $\{\mu_{s,t,q}\}$ indexed by $\Delta$.
\end{thm}

\begin{proof}
Lemma~\ref{may27-8} defines a map $\frC_3\to  \frC_2$.
For fixed $[[M_{\Vd},M_\Ud]_G]\in \calC_2$,
choose a representative 
$(M_V,M_U)$,
where $M_V=(m_{\Vd,ij})$ and $M_U=(m_{\Ud,ij})$.
For each $(s,t,q)\in \Delta$,
    we define $\mu_{s,t,q}$
    as the component $H^s(a_s)\to H^t(q)$ of $m_{Z_s,(w_t-2q)(w_s-2a_s)}$.
    By the uniqueness in Lemma~\ref{may27-8},
    the pair associated with $(G,\{\mu_{s,t,q}\}_{(s,t,q)\in \Delta})\in \frC_3$
    is precisely $(M_{\Vd},M_\Ud)$.
    Hence the map is surjective.
\end{proof}

\subsection{The case without eigenvalue $1$}\label{ssclasmhmnonunip}

Let us finally consider monodromic mixed Hodge modules $\calM\in \MHM_{\mon,\neq 1}(\CC_t)$.
Throughout this subsection, we write $T$ for the semisimple part of the monodromy.
As explained in Subsection~\ref{aug5-1},
classifying such mixed Hodge modules is equivalent to classifying diagrams of the form
\begin{align}\label{jun15-5}
    \xymatrix@C=40pt{(\Ud,T)\ar@<0.5ex>[r]^-{-N,\ (-1)}&(\Ud,T)\ar@<0.5ex>[l]^-{=}},
\end{align}
where $(\Ud,T)$ is a mixed Hodge structure with a finite-order automorphism $T$ having no eigenvalue $1$, and
$N$ is a nilpotent morphism $N\colon \Ud\to \Ud(-1)$ which commutes with $T$.
Then,
$T$ and $N$ induce maps
$\GR^W_\bullet T$ and $\GR^W_\bullet N$ on $\GR^W_\bullet \Ud$, respectively.
Similarly to (\ref{mar4-2}),
we define a diagram
\begin{align}\label{jun15-15}
    \xymatrix@C=40pt{(\GR^{W}_\bullet\Ud,\GR^W_\bullet  T)\ar@<0.5ex>[r]^-{-\GR^W_\bullet N,\ (-1)}&(\GR^W_\bullet \Ud,\GR^W_\bullet T)\ar@<0.5ex>[l]^-{=}},
\end{align}
and denote the corresponding mixed Hodge module by $G(M)$.
As we already mentioned in Subsection~\ref{aug5-2},
we have $\Vd_i=\Vd'_i=\Ud_i=0$,
and $T$ induces an automorphism on $\Ud'_i$,
which we will denote again by $T$.

Since $T$ is a finite-order automorphism, 
Lemma~\ref{jun13-8} implies that
the proof of Theorem~\ref{may15-5} extends to the $T$-equivariant setting.
To state the corresponding form of Theorem~\ref{may15-5} in the present situation, we introduce the following notation.

Let $H$ be a pure Hodge structure equipped with a finite-order automorphism $T$.
For $i\geq 1$, we denote by $\calA_i(H,T)$ the zero extension to $\CC$ of the mixed Hodge module on $\CS$ corresponding to the tensor product of $A_i^H|_{\CS}$ and the variation of Hodge structure on $\CS$ with fiber $H$ and monodromy $T$,
where $A_i^H$ is defined in Definition~\ref{ap8-0}.
This corresponds to
the diagram:
\begin{align}\label{jun15-10}
    \xymatrix@C=40pt{\bigoplus_{j=0}^{i-1}(H(j),T(j))\ar@<0.5ex>[r]^-{-N,\ (-1)}&\bigoplus_{j=0}^{i-1}(H(j),T(j))\ar@<0.5ex>[l]^-{=}},
\end{align}
where $-N$ is given by the identity map $H(j)\to H(j+1)(-1)$ for $0\leq j\leq i-2$, and by zero on $H(i-1)$.

\begin{cor}
    The mixed Hodge module $G(M)$ can be decomposed as
    \begin{align}\label{jun16-1}
        G(M)\simeq \bigoplus_{s=1}^{\ell}\calA_{i(s)}(H^s,T),
    \end{align}
where 
$i(s)\geq 1$ and $(H^s,T)$ is a pure Hodge structure with a finite-order automorphism $T$ without eigenvalue $1$.
\end{cor}

We call the object appearing on the right-hand side of (\ref{jun16-1}), equipped with the direct sum decomposition,
or the corresponding diagram with the induced decomposition, a split graded model.
For such a model, the mixed Hodge structures 
$L_{\Ud,w}$ are defined in the same way as in (\ref{may15-6}), and each $L_{\Ud,w}$ is equipped with a finite-order automorphism $T$.
We denote by $\Exten{\{L_{\Ud,w}\}_{w\in \Lambda}}^{\equiva}$
the set of isomorphism classes of 
mixed Hodge structures $H$ equipped with a finite-order automorphism $T$,
together with isomorphisms 
\[(\GR^W_wH, \GR^W_w T)\simeq (L_{\Ud,w},T)\qquad (w\in \Lambda).\]
We define the $T$-equivariant subgroup $\UpT^{T}$ by
\[\UpT^{T}:=\{(m_{ij})_{i,j\in \Lambda}\in \UpT\ |\ Tm_{ij}=m_{ij}T\}.\]
We define $\UpT^{T}_{\QQ}$ and $\UpT^{T}_{F}$ in the same way. 
The following is the $T$-equivariant version of Proposition~\ref{jun16-2}.
\begin{lem}
There is a natural bijection
\[\UpT_{\QQ}^{T}\backslash \UpT^{T}/\UpT^{T}_F\simto {\Exten{\{L_{\Ud,w}\}_{w\in \Lambda}}}^{\equiva}.\]    
\end{lem}

Since $T$ on $L_{\Ud,w}$ is a direct sum of morphisms of Hodge structures,
the proof of Lemma~\ref{may7-1} also works $T$-equivariantly.
Proposition~\ref{sep28-1} also holds in the $T$-equivariant setting.

Replacing the split graded models and $\UpT$ in the definition of $\calC$ by their $T$-equivariant counterparts,
we define a category $\calC^\equiva$.
Then, we have an equivalence
\[\calC^{\equiva}\simto \MHM_{\mon,\neq 1}(\CC_t).\]
Similarly,
imposing $T$-equivariance on all the data and morphisms in
Definition~\ref{may21-3} and Definition~\ref{may22-1},
we obtain the sets
$\frC_1^{\equiva}(G)$, $\frC_2^\equiva$ and $\frC_3^\equiva$.
The analogues of Proposition~\ref{may27-9} and
Theorem~\ref{aug5-5} for $\MHM_{\mon,\neq 1}(\CC_t)$ then follow from the same arguments.